\documentclass{amsproc}

\usepackage{amssymb}
\usepackage{amsbsy}
\usepackage{amscd}
\usepackage{amsmath}
\usepackage{amsthm}
\usepackage[mathscr]{eucal}
\usepackage[all]{xy}
\usepackage[colorlinks,plainpages,backref,urlcolor=blue]{hyperref}

\newtheorem{theorem}{Theorem}[section]
\newtheorem{lemma}[theorem]{Lemma}
\newtheorem{corollary}[theorem]{Corollary}
\newtheorem{prop}[theorem]{Proposition}

\theoremstyle{definition}
\newtheorem{definition}[theorem]{Definition}
\newtheorem{example}[theorem]{Example}
\newtheorem{remark}[theorem]{Remark}

\newcommand{\Z}{\mathbb{Z}}
\newcommand{\Q}{\mathbb{Q}}
\newcommand{\R}{\mathbb{R}}
\newcommand{\C}{\mathbb{C}}

\newcommand{\QP}{\mathbb{QP}}
\newcommand{\CP}{\mathbb{CP}}
\newcommand{\RP}{\mathbb{RP}}
\newcommand{\bP}{\mathbb{P}}
\renewcommand{\AA}{\mathbf{A}}

\newcommand{\sV}{\mathsf{V}}
\newcommand{\sW}{\mathsf{W}}
\newcommand{\sE}{\mathsf{E}}
\newcommand{\cS}{\mathsf{p}}

\newcommand{\G}{\Gamma}

\renewcommand{\k}{\Bbbk}
\newcommand{\RR}{\mathcal{R}}
\newcommand{\VV}{\mathcal{V}}
\newcommand{\A}{{\mathcal{A}}}

\renewcommand{\L}{{\mathcal{L}}}
\newcommand{\CC}{{\mathcal{C}}}
\newcommand{\WW}{\mathcal{W}}

\DeclareMathOperator{\rank}{rank}
\DeclareMathOperator{\gr}{gr}
\DeclareMathOperator{\im}{im}

\DeclareMathOperator{\codim}{codim}

\DeclareMathOperator{\id}{id}
\DeclareMathOperator{\ab}{{ab}}

\DeclareMathOperator{\Sym}{Sym}
\DeclareMathOperator{\ch}{char}

\DeclareMathOperator{\Hom}{{Hom}}

\DeclareMathOperator{\ev}{ev}

\DeclareMathOperator{\FF}{{F}}

\DeclareMathOperator{\init}{in}
\DeclareMathOperator{\lk}{lk}

\DeclareMathOperator{\dire}{dir}

\DeclareMathOperator{\Grass}{Gr}

\DeclareMathOperator{\TC}{TC}

\DeclareMathOperator{\linn}{\,lin}

\newcommand{\PL}{\scriptscriptstyle{\rm PL}}

\newcommand{\wG}{\widehat{G}}

\newcommand{\same}{\Longleftrightarrow}
\newcommand{\surj}{\twoheadrightarrow}
\newcommand{\inj}{\hookrightarrow}

\newcommand{\isom}{\xrightarrow{\simeq}}
\newcommand{\compl}{\scriptscriptstyle{\complement}}

\newcommand{\abs}[1]{\left| #1 \right|}

\def\set#1{{\{\, #1 \,\}}}

\newcommand{\bigmid}{\:\big|  \big.\:}
\newcommand{\pt}{\{\text{pt}\}}

\newenvironment{alphenum}
{

\begin{enumerate}}{\end{enumerate}}

\title[Resonance varieties and Dwyer--Fried invariants]%
{Resonance varieties and Dwyer--Fried invariants}

\author[Alexander~I.~Suciu]{Alexander~I.~Suciu}

\address{Department of Mathematics,
Northeastern University,
Boston, MA 02115, USA}
\urladdr{http://www.math.neu.edu/\~{}suciu}
\email{a.suciu@neu.edu}

\thanks{Partially supported by NSA grant H98230-09-1-0021 
and NSF grant DMS--1010298}

\subjclass[2000]{Primary 
20J05,  
55N25; 
Secondary 
14F35, 
32S22,  
55R80,  
57M07.  
}

\keywords{Free abelian cover, characteristic variety, resonance 
variety, tangent cone, Dwyer--Fried set, special 
Schubert variety, toric complex, 
K\"{a}hler manifold, hyperplane arrangement}

\begin{document}

\begin{abstract}
The Dwyer--Fried invariants of a finite cell complex $X$ 
are the subsets $\Omega^i_r(X)$ of the Grassmannian of 
$r$-planes in $H^1(X,\Q)$ which parametrize the 
regular $\Z^r$-covers of $X$ having finite 
Betti numbers up to degree~$i$. In previous 
work, we showed that each $\Omega$-invariant 
is contained in the complement of a union of Schubert 
varieties associated to a certain subspace arrangement 
in $H^1(X,\Q)$. Here, we identify a class of spaces for 
which this inclusion holds as equality.  For such 
``straight" spaces $X$, all the data required to compute 
the $\Omega$-invariants can be extracted from the 
resonance varieties associated to the cohomology 
ring $H^*(X,\Q)$. In general, though, translated components 
in the characteristic varieties affect the answer. 
\end{abstract}

\maketitle
\tableofcontents

\section{Introduction}
\label{sec:intro}

One of the most fruitful ideas to arise from arrangement 
theory is that of turning the cohomology ring of a space 
into a family of cochain complexes, parametrized by   
the cohomology group in degree $1$, and extracting certain 
varieties from these data, as the loci where 
the cohomology of those cochain complexes jumps.  
What makes these ``resonance" varieties really useful is their 
close connection with a different kind of jumping 
loci: the ``characteristic" varieties, which record the 
jumps in homology with coefficients in rank $1$ local systems. 

In this paper, we use the geometry of the cohomology 
jump loci to study a classical problem in topology: 
determining which infinite covers of a 
space have finite Betti numbers.  Restricting our 
attention to regular, free abelian covers of a fixed rank 
allows us to state the problem in terms of a suitable filtration 
on the rational Grassmannian. Under favorable circumstances, 
the finiteness of the Betti numbers of such covers 
is exclusively controlled by the incidence varieties to the 
resonance varieties of our given space.

\subsection{Cohomology jump loci and straightness}
\label{subsec:intro straight}

Let $X$ be a connected CW-complex with finite skeleta.  
To such a space, we associate two types of jump loci. 
The first are the {\em resonance varieties}\/ $\RR^i(X)$.   
These are homogeneous subvarieties of the affine space 
$H^1(X,\C)=\C^n$, where $n=b_1(X)$, and they are defined in 
terms of the cohomology algebra $A=H^*(X,\C)$, as follows. 
For each $a\in A^1$, left-multiplication by $a$ defines a 
cochain complex $(A,\cdot a)$.  Then 
\begin{equation}
\label{eq:intro res}
\RR^i(X)=\{a \in H^1(X,\C) \mid 
 H^j(A,\cdot a) \ne 0, \text{ for some $j\le i$}\}. 
\end{equation}

The second type of jump loci we consider here are the 
{\em characteristic varieties}\/ $\WW^i(X)$. These are 
Zariski closed subsets of the complex algebraic torus  
$H^1(X,\C^{\times})^0=\Hom(\pi_1(X),\C^{\times})^0=(\C^{\times})^n$, 
defined as follows. Each character $\rho\colon \pi_1(X)\to \C^{\times}$ 
gives rise to a rank~$1$ local system on $X$, call it $\L_{\rho}$.  
Then 
\begin{equation}
\label{eq:intro cv}
\WW^i(X)=\{\rho \in H^1(X,\C^{\times})^0 \mid 
H_j(X, \L_{\rho})\ne 0, \text{ for some $j\le i$}\}. 
\end{equation}

One of our goals in this paper is to isolate a class of spaces 
for which the resonance and characteristic varieties 
have a rather simple nature, and are intimately related 
to each other. 

We say that $X$ is {\em locally $k$-straight}\/ if, for each 
$i\le k$, all components of $\WW^{i}(X)$ passing through 
the origin $1$ are algebraic subtori, and the tangent cone 
at $1$ to $\WW^i(X)$ equals $\RR^{i}(X)$. If, moreover, 
all positive-dimensional components of $\WW^{i}(X)$ 
contain the origin, we say $X$ is {\em $k$-straight}. 
For locally straight spaces, the resonance varieties $\RR^i(X)$
are finite unions  of rationally defined linear subspaces. 

Examples of straight spaces include Riemann surfaces, 
tori, and knot complements. Under some further assumptions, 
the straightness properties behave well with respect 
to finite direct products and wedges.

A related notion is Sullivan's $k$-formality.  Using the 
tangent cone formula from \cite{DPS-duke}, 
it is readily seen that $1$-formal spaces are 
locally $1$-straight.  In general though, $1$-formality  
does not imply $1$-straightness, and the converse 
does not hold, either.

\subsection{Dwyer--Fried invariants}
\label{subsec:intro df}

The second goal of this paper is to analyze the homological 
finiteness properties of all regular, free abelian covers of a 
given space $X$, and relate these properties to the resonance 
varieties of $X$, under a straightness assumption. 

The connected, regular $\Z^r$-covers $\hat{X}\to X$ are 
parametrized by the Grassmannian of $r$-planes in the 
vector space $H^1(X,\Q)$. Moving about this variety, and 
recording when all the Betti numbers 
$b_1(\hat{X}),\dots, b_i(\hat{X})$ are finite defines subsets 
\begin{equation}
\label{eq:intro df}
\Omega^i_r(X)\subset \Grass_r(H^1(X,\Q)),
\end{equation}
which we call the {\em Dwyer-Fried invariants}\/ of $X$. 
These sets depend only on the homotopy type of $X$. 
Consequently, if $G$ is a  finitely generated group, 
the sets $\Omega^i_r(G):=\Omega^i_r(K(G,1))$ 
are well-defined.

In \cite{DF}, Dwyer and Fried showed that the support 
varieties of the Alexander invariants of a finite cell 
complex $X$ completely determine the $\Omega$-sets of $X$.  
In \cite{PS-plms} and \cite{Su}, this foundational 
result was refined and reinterpreted in terms of the 
characteristic varieties of $X$, as follows.  Let $\exp\colon 
H^1(X,\C)\to H^1(X,\C^{\times})$ be the coefficient 
homomorphism induced by the exponential map 
$\exp\colon \C\to \C^{\times}$. Then, 
\begin{equation}
\label{eq:intro grass cv}
\Omega^i_r(X)=\big\{P\in \Grass_r(\Q^n) \bigmid
\dim_{\C} ( \exp(P \otimes \C) \cap \WW^i(X)) = 0 \big\}.
\end{equation}

We pursue this study here, by investigating the relationship 
between the Dwyer--Fried sets and the resonance varieties.  
Given a homogeneous variety $V\subset \k^n$, let 
$\sigma_r(V)\subset \Grass_r(\k^n)$ be the variety of 
$r$-planes incident to $V$.  Our main result reads as follows. 

\begin{theorem}
\label{thm:intro main}
Let $X$ be a connected CW-complex with finite $k$-skeleton. 
\begin{enumerate}
\item \label{main1}
Suppose $X$ is locally $k$-straight. Then, for all $i\le k$ 
and $r\ge 1$, 
\[
\Omega^i_r(X) \subseteq \Grass_r(H^1(X,\Q)) \setminus 
\sigma_r(\RR^i(X,\Q)).
\]
\item \label{main2}
Suppose $X$ is $k$-straight. Then, for all $i\le k$ and $r\ge 1$,
\[
\Omega^i_r(X) = \Grass_r(H^1(X,\Q)) \setminus 
\sigma_r(\RR^i(X,\Q)).
\]
\end{enumerate}
\end{theorem}

As a consequence, if $X$ is $k$-straight, then 
each set $\Omega^i_r(X)$ with $i\le k$ is the 
complement of a finite union of special Schubert 
varieties in the Grassmannian of $r$-planes in $\Q^n$, 
where $n=b_1(X)$. In particular, $\Omega^i_r(X)$ is 
a Zaris\-ki open subset of $\Grass_r(\Q^n)$.

\subsection{Applications}
\label{subsec:intro apps} 

We illustrate our techniques with a broad variety of 
examples, coming from low-dimensional 
topology, toric topology, algebraic geometry, 
and the theory of hyperplane arrangements. 

One class of spaces for which things work out very 
well are the toric complexes.  Every simplicial 
complex $K$ on $n$ vertices determines a subcomplex 
$T_K$ of the $n$-torus, with fundamental group the 
right-angled Artin group associated 
to the $1$-skeleton of $K$.  It turns out that all 
toric complexes are straight (that is, $k$-straight, 
for all $k$).  As shown by Papadima and Suciu in \cite{PS-adv}, 
the resonance varieties of a toric complex are unions 
of coordinate subspaces, which can be read off 
directly from the corresponding simplicial complex.  
This leads to some explicit formulas for the 
$\Omega$-sets of toric complexes and 
right-angled Artin groups. 

The characteristic varieties of (quasi-) K\"{a}hler manifolds 
are fairly well understood, due to work of Beauville, 
Green--Lazarsfeld, Simpson, Campana, and finally 
Arapura \cite{Ar}.  In particular, if $X$ is a smooth, 
quasi-projective variety, then all the components of 
$\WW^1(X)$ are torsion-translated subtori.  Recent work 
of Dimca, Papadima, and Suciu \cite{DPS-duke} sheds 
light on the first resonance variety of such varieties:  
if  $X$ is also $1$-formal (e.g., if $X$ is compact), 
then all the components of $\RR^1(X)$ are rationally 
defined linear subspaces. It follows that $X$ is locally 
$1$-straight, and 
$\Omega^1_r(X) \subseteq \sigma_r(\RR^1(X,\Q))^{\compl}$.  
In general, though, the inclusion can be strict.  For 
instance, we prove in  Theorem \ref{thm:tt} the following: 
if $\WW^1(X)$ has a $1$-dimensional 
component not passing through $1$, and $\RR^1(X)$ has 
no codimension-$1$ components, then 
$\Omega^1_2(X) \ne \sigma_2(\RR^1(X,\Q))^{\compl}$. 

Hyperplane arrangements have been the main driving 
force behind the development of the theory of cohomology 
jump loci, and still provide a rich source of motivational 
examples for this theory.  If $\A$ is an arrangement 
of hyperplanes in $\C^{\ell}$, then its   
complement, $X=\C^{\ell}\setminus \bigcup_{H\in \A} H$,
is a connected, smooth, quasi-projective variety.  
It turns out that $X$ is also  
formal, locally straight, but not always straight. 
The first resonance variety of the arrangement 
is completely understood, owing to work of 
Falk \cite{Fa97}, Cohen, Libgober, Suciu, Yuzvinsky, 
and others, with the state of the art being the recent 
work of Falk, Pereira, and Yuzvinsky \cite{FY, PeY, Yu}. 
This allows for explicit computations of the 
Dwyer--Fried invariants of various classes of arrangements.  
For the deleted $\operatorname{B}_3$ arrangement, 
though, the computation is much more subtle, 
due to the presence of a translated component 
in the first characteristic variety. 

\subsection{Organization of the paper}
\label{subsec:intro org} 

The paper is divided in roughly two parts.  In the first part 
(Sections \ref{sec:aomoto}--\ref{sec:straight}), we recall 
some of the basic theory of cohomology jump loci, and 
develop the notion of straightness. In the second part 
(Sections \ref{sec:df res}--\ref{sec:arrs}), we develop the 
Dwyer--Fried theory in the straight context, and apply it 
in a variety of settings.

In \S\ref{sec:aomoto} we define the Aomoto complex 
of a space $X$, and study it in more detail in the case 
when $X$ admits a minimal cell structure. We use 
the Aomoto complex in \S\ref{sec:res vars} to define 
the resonance varieties, and establish some 
of the basic properties of these varieties.

In \S\ref{sec:tcones vars} we define two types of tangent 
cones to a subvariety of $(\C^{\times})^n$, and 
recall some of their features.  In \S\ref{sec:cvs}, we introduce 
the characteristic varieties, and review some  
pertinent facts about their tangent cones.  Finally, in 
\S\ref{sec:straight} we define and study the various 
notions of straightness, based on the geometry of the 
jump loci.

We start \S\ref{sec:df res} with a review of the Dwyer--Fried 
invariants, and the way they relate to the characteristic varieties, 
after which we prove Theorem \ref{thm:intro main}. 
In \S\ref{sec:formality}, we discuss the relevance of 
formality in this setting. 

Finally, we show how our techniques work for three 
classes of spaces:  toric complexes in \S\ref{sec:toric}, 
K\"{a}hler and quasi-K\"{a}hler manifolds in \S\ref{sec:kahler}, 
and complements of hyperplane arrangements in \S\ref{sec:arrs}.  
In each case, we explain what is known about the 
cohomology jump loci of those spaces, and 
use this knowledge to determine their straightness 
properties, and to compute some of their $\Omega$-invariants.

\section{The Aomoto complex}
\label{sec:aomoto}

We start by recalling the definition of the (universal) 
Aomoto complex associated to the cohomology ring 
of a space $X$. 
When $X$ admits a minimal cell structure, this cochain 
complex can be read off from the equivariant cellular 
chain complex of the universal abelian cover, $X^{\ab}$. 

\subsection{A cochain complex from the cohomology ring}
\label{subsec:aomoto}

Let $X$ be a connected CW-complex with finite $k$-skeleton, 
for some $k\ge 1$. Consider the cohomology algebra 
$A=H^* (X,\C)$, with product operation given by the 
cup product of cohomology classes.   For each $a\in A^1$, 
we have $a^2=0$, by graded-commutativity of the cup product. 
\begin{definition}
\label{def:aom}
The {\em Aomoto complex}\/ of $A$ (with respect to $a\in A^1$) 
is the cochain 
complex of finite-dimensional, complex vector spaces,
\begin{equation}
\label{eq:aomoto}
\xymatrixcolsep{18pt}
\xymatrix{(A , a)\colon  \ 
A^0\ar^(.66){a}[r] & A^1\ar^(.45){a}[r] 
& A^2  \ar^(.45){a}[r]& \cdots },
\end{equation}
with differentials given by left-multiplication by $a$.
\end{definition}

Here is an alternative in\-terpretation. Pick a basis 
$\{ e_1,\dots, e_n \}$ for the complex vector space 
$A^1=H^1(X,\C)$, and let $\{ x_1,\dots, x_n \}$ 
be the Kronecker dual basis for $A_1=H_1(X,\C)$.  
Identify the symmetric algebra $\Sym(A_1)$ 
with the polynomial ring $S=\C[x_1,\dots, x_n]$. 

\begin{definition}
\label{def:univ aom}
The {\em universal Aomoto complex}\/ of $A$ is the 
cochain complex of free $S$-modules, 
\begin{equation}
\label{eq:univ aomoto}
\xymatrixcolsep{18pt}
\AA\colon 
\xymatrix{
\cdots \ar[r] 
&A^{i}\otimes_{\C} S \ar^(.45){d^{i}}[r] 
&A^{i+1} \otimes_{\C} S \ar^(.5){d^{i+1}}[r] 
&A^{i+2} \otimes_{\C} S \ar[r] 
& \cdots},
\end{equation}
where the differentials are defined by 
$d^{i}(u \otimes 1)= \sum_{j=1}^{n} 
e_j u \otimes x_j$ for $u\in A^{i}$, and then 
extended by $S$-linearity. 
\end{definition}
The fact that $\AA$ is a cochain complex is  
verified as follows:
\begin{align*}
d^{i+1}d^i (u\otimes 1) &=\sum_{k=1}^n \sum_{j=1}^{n} 
e_ke_j u \otimes x_j x_k \\ 
&= \sum_{j<k} (e_k e_j +e_j e_k) u \otimes x_j x_k =0.
\end{align*}

The relationship between the two definitions is given 
by the following well-known lemma.
\begin{lemma}
\label{lem:two aom}
The evaluation of the universal Aomoto complex $\AA$ 
at an element $a\in A^1$ 
coincides with the Aomoto complex $(A,a)$. 
\end{lemma}
\begin{proof}
Write $a=\sum_{j=1}^{n} a_j e_j \in A^1$, and 
let $\ev_a \colon S\to \C$ be the ring morphism given 
by $g\mapsto g(a_1,\dots, a_n)$.  The resulting 
cochain complex, $\AA (a)= \AA\otimes_S \C$, 
has differentials $d^i(a):=d^i \otimes \id_{\C}$ 
given by
\begin{equation}
\label{eq:deltai}
d^i(a)(u)=\sum_{j=1}^{n} e_j u \otimes \ev_a(x_j) 
= \sum_{j=1}^{n} e_j u\cdot  a_j = a\cdot u. 
\end{equation}
Thus, $\AA (a)=(A,a)$.
\end{proof}

\subsection{Minimality and the Aomoto complex}
\label{subsec:min aomoto}

As shown by Papadima and Suciu in \cite{PS-tams}, 
the universal Aomoto complex of $H^*(X,\C)$ is functorially 
determined by the equivariant chain complex of the universal 
abelian cover $X^{\ab}$, provided $X$ admits a minimal 
cell structure.

More precisely, suppose $X$ is a connected, 
finite-type CW-complex.  
We say the CW-structure on $X$ is {\em minimal}\/ 
if the number of $i$-cells of $X$  equals the 
Betti number $b_i(X)$, for every $i\ge 0$. 
Equivalently, the boundary maps in the cellular 
chain complex $C_{\bullet}(X,\Z)$ are all zero maps. 
In particular, the homology groups $H_i(X,\Z)$ are all 
torsion-free.

\begin{theorem}[\cite{PS-tams}]
\label{thm:linaom}
Let $X$ be a minimal CW-complex. Then the linearization 
of the cochain complex $C^{\bullet}(X^{\ab},\C)$ 
coincides with the universal Aomoto complex of 
$H^*(X,\C)$.
\end{theorem}

Let us explain in more detail how this theorem works. 
Pick an isomorphism $H_1(X,\Z)\cong\Z^n$, and 
identify $\C[\Z^n]$ with  
$\Lambda=\C[t_1^{\pm 1}, \dots , t_n^{\pm 1}]$.  
Next, filter $\Lambda$ by powers of the maximal ideal 
$I=(t_1-1,\dots,t_n-1)$, and identify the associated graded 
ring, $\gr(\Lambda)$, with the polynomial ring 
$S=\C[x_1,\dots,x_n]$, via the ring map $t_i-1\mapsto x_i$.  

The minimality hypothesis allows us to identify 
$C_{i} (X^{\ab}, \C)$ with $\Lambda \otimes_{\C} H_i(X,\C)$   
and $C^{i} (X^{\ab}, \C)$ with $A^{i} \otimes_{\C} \Lambda$.  
Under these identifications, the boundary map 
$\partial_{i+1}^{\ab}\colon C_{i+1} (X^{\ab}, \C)
\to C_{i} (X^{\ab}, \C) $ dualizes to a map $\delta^i\colon 
A^i \otimes_{\C} \Lambda \to A^{i+1} \otimes_{\C} \Lambda$. 
Let $\gr(\delta^i)\colon 
A^i \otimes_{\C} S \to A^{i+1} \otimes_{\C} S$ 
be the associated graded of $\delta^i$, and let 
$\gr(\delta^i)^{\linn}$ be its linear part.   
Theorem \ref{thm:linaom} then provides an identification
\begin{equation}
\label{eq:linaom}
\gr(\delta^i)^{\linn}=d^i\colon 
A^i \otimes_{\C} S \to A^{i+1} \otimes_{\C} S,
\end{equation}
for each $i\ge 0$.

\begin{example}
\label{ex:koszul}
Let $X=T^n$ be the $n$-dimensional torus, with the 
standard product cell structure.  Then $X$ is a minimal 
cell complex, and $C_{\bullet}(X^{\ab},\C)$ is  
the Koszul complex $K(t_1-1,\dots , t_n-1)$ over the 
ring $\Lambda$.  The cohomology ring 
$H^*(T^n,\C)$ is the exterior algebra $E$ on variables 
$e_1,\dots, e_n$, and the universal Aomoto complex 
$E\otimes_{\C} S$ is the Koszul complex $K(x_1,\dots ,x_n)$ 
over the ring $S$.  In this case, Theorem \ref{thm:linaom} 
simply says that the substitution $t_i-1\mapsto x_i$ takes 
one Koszul complex to the other.
\end{example}

\begin{example}
\label{ex:s1s2}
Let $Y=S^1 \vee S^2$, and identify $\pi_1(Y)=\Z$, 
with generator $t$, and $\pi_2(Y)=\Z[t^{\pm 1}]$.    
Given a polynomial $f\in \Z[t]$, let $\varphi\colon S^2 \to Y$ 
be a map representing $f$, and attach a $3$-cell to $Y$ 
along $\varphi$ to obtain a CW-complex $X_f$.  
For instance, if $f(t)=t-1$, then $X_f \simeq S^1\times S^2$.   
More generally, $X_f$ is minimal if and only if $f(1)=0$, in 
which case $H_*(X_f,\Z)\cong H_*(S^1\times S^2,\Z)$.

Now identify $\pi_1(X_f)=\Z$ and $\C[\Z]$ with 
$\Lambda=\C[t^{\pm 1}]$.  The chain complex 
$C_{\bullet}(X^{\ab}_f,\C)$ can then be written as 
\begin{equation}
\label{eq:s1s2ch}
\xymatrix{\Lambda \ar^{f(t)}[r] & \Lambda \ar^{0}[r] &   
\Lambda \ar^{t-1}[r] &  \Lambda}.
\end{equation}

Finally, identify $\gr(\Lambda)$ with $S=\C[x]$, and set 
$g(x)=f(1+x)$.  Suppose $f(1)=0$, so that $X_f$ is minimal.  
Then, the linear term of $g(x)$ is $f'(1)\cdot x$, and 
so the universal Aomoto complex of $H^*(X_f,\C)$ is
\begin{equation}
\label{eq:s1s2aom}
\xymatrixcolsep{26pt}
\xymatrix{S \ar^{x}[r] & S \ar^{0}[r] &   
S \ar^{f'(1)x}[r] &  S}.
\end{equation}
\end{example}

\section{Resonance varieties}
\label{sec:res vars}

In this section, we review the definition and basic properties 
of the resonance varieties of a space $X$, which measure 
the deviation from exactness of the Aomoto complexes 
associated to the cohomology ring $H^*(X,\C)$.  

\subsection{Jump loci for the Aomoto-Betti numbers}
\label{subsec:res vars}

As usual, let $X$ be a connected CW-complex with 
finite $k$-skeleton.  Denote by $A$ the cohomology 
algebra $H^*(X,\C)$.  
Computing the homology of the Aomoto complexes 
$(A,a)$ for various values of the parameter $a\in A^1$, 
and recording the resulting Betti numbers, carves 
out some very interesting subsets of the affine space 
$A^1=\C^n$, where $n=b_1(X)$.

\begin{definition}
\label{def:rvs}
The {\em resonance varieties}\/ of $X$ are the sets 
\begin{equation*}
\label{eq:rvs}
\RR^i_d(X)=\{a \in A^1 \mid 
\dim_{\C} H^i(A,\cdot a) \ge  d\},  
\end{equation*}
defined for all integers $0\le i\le k$ and $d>0$. 
\end{definition}

The degree-$1$, depth-$1$ resonance variety is especially 
easy to describe: $\RR^1_1(X)$ consists of those elements 
$a\in A^1$ for which there exists an element $b \in A^1$, 
not proportional to $a$, and such that $b\cdot a=0$. 

The terminology from Definition \ref{def:rvs} is justified 
by the following well-known lemma.  For completeness, 
we include a proof.

\begin{lemma}
\label{lem:res ideal}
The sets $\RR^i_d(X)$ are homogeneous algebraic 
subvarieties of the affine space $A^1=H^1(X,\C)$.   
\end{lemma}

\begin{proof}  
By definition, an element $a\in A^1$ belongs to 
$\RR^i_d(X)$ if and only if 
$\rank \delta^{i-1}(a) + \rank \delta^{i}(a) \le c_i -d$, 
where $c_i=c_i(X)$ is the number of $i$-cells of $X$.  
As a set, then, $\RR^i_d(X)$ can be written as the intersection 
\begin{equation*}
\label{eq:big int}
\bigcap_{\stackrel{r+s=c_{i}-d+1}{r,s\ge 0}}
\{ a\in A^1 \mid \rank  \delta^{i-1}(a)  \le r-1 
\text{ or } \rank  \delta^{i}(a)  \le s-1\}.
\end{equation*}
Using this description, we may rewrite $\RR^i_d(X)$ as 
the zero-set of a sum of products of determinantal ideals,
\begin{equation}
\label{eq:ideal res}
\RR^i_d(X) = V\Bigg( \sum_{p+q=c_{i+1}+d-1} 
E_p(\delta^{i-1}) \cdot E_q(\delta^{i}) \Bigg).
\end{equation}

Clearly, $a\in \RR^i_d(X)$ if and only if $\lambda a \in \RR^i_d(X)$, 
for all $\lambda\in \C^{\times}$.  Thus, $\RR^i_d(X)$ is a 
homogeneous variety.
\end{proof}

Sometimes it will be more convenient to consider 
the projectivization $\overline{\RR}^i_d(X)$, 
viewed as a subvariety of $\bP(A^1)=\CP^{n-1}$. 

The resonance varieties $\RR^i_d(X)$ are homotopy-type 
invariants of the space $X$.  The following (folklore) result 
makes this more precise.

\begin{lemma}
\label{lem:res inv}
Suppose $X\simeq X'$.   There is then a linear isomorphism 
$H^1(X',\C)\cong H^1(X,\C)$ which restricts to isomorphisms 
$\RR^i_d(X') \cong \RR^i_d(X)$, for all $i\le k$ and $d>0$.  
\end{lemma}

\begin{proof}
Let $f\colon X\to X'$ be a homotopy equivalence.  
The induced homomorphism in cohomology,  
$f^* \colon H^*(X',\C)\isom H^*(X,\C)$,  defines 
isomorphisms $(H^*(X',\C),a)\isom (H^*(X,\C),f^*(a))$ 
between the respective Aomoto complexes, for all 
$a\in H^1(X',\C)$. Hence, 
$f^* \colon H^1(X',\C)\isom H^1(X,\C)$ restricts to 
isomorphisms $\RR^i_d(X') \isom \RR^i_d(X)$.  
\end{proof}

\subsection{Discussion}
\label{subsec:discuss}

In each degree $i\ge 0$, the 
resonance varieties provide a descending filtration,
\begin{equation}
\label{eq:res filt}
H^1(X,\C)= \RR^i_0(X)  \supseteq \RR^i_1(X) 
\supseteq \RR^i_2(X) \supseteq \cdots .
\end{equation} 

Note that $0\in \RR^i_d(X)$, for $d\le b_i(X)$, but 
$\RR^i_d(X)=\emptyset$, for $d>b_i(X)$. 
In degree $0$, we have $\RR^0_1(X)= \{ 0\}$,
and $\RR^0_d(X)= \emptyset$, for $d>1$. In degree $1$, 
the varieties $\RR^1_d(X)$ depend only on the group 
$G=\pi_1(X,x_0)$---in fact, only on the cup-product map 
$\cup \colon H^1(G,\C) \wedge H^1(G,\C) \to H^2(G,\C)$.   
Accordingly, we will sometimes write 
$\RR^1_d(G)$ for $\RR^1_d(X)$.  

\begin{example}
\label{ex:res torus}
Let $T^n$ be the $n$-dimensional torus.  Using the 
exactness of the Koszul complex from Example \ref{ex:koszul}, 
we see that $\RR^i_d(T^n)$ equals $\{0\}$ if $d\le \binom{n}{i}$, 
and is empty, otherwise. 
\end{example}

\begin{example}
\label{ex:res surf}
Let $S_g$ be the compact, connected, orientable 
surface of genus $g>1$.  With suitable identifications 
$H^1(S_g,\C)=\C^{2g}$ and $H^2(S_g,\C)=\C$, the 
cup-product map 
$\cup \colon H^1(S_g,\C) \wedge H^1(S_g,\C) \to H^2(S_g,\C)$
is the standard symplectic form.  A computation then 
shows
\begin{equation}
\label{eq:res surf}
\RR^i_d(S_g)=
\begin{cases}
\C^{2g} & \text{if $i=1$, $d< 2g-1$},\\
\{0\}& \text{if $i=1$, $d\in \{2g-1, 2g\}$, or $i\in\{0, 2\}$, $d=1$},\\
\emptyset & \text{otherwise}.
\end{cases}
\end{equation} 
\end{example}

One may extend the definition of resonance varieties 
to arbitrary fields $\k$, provided $H_1(X,\Z)$ is torsion-free, 
if $\ch\k=2$. The resulting varieties, $\RR^i_d(X,\k)$, 
behave well under field extensions: 
if $\k\subseteq \mathbb{K}$, then
$\RR^i_d(X,\k)=\RR^i_d(X,\mathbb{K}) \cap H^1(X,\k)$. 
In particular, $\RR^i_d(X,\Q)$ is just the set of rational 
points on the integrally defined variety 
$\RR^i_d(X)=\RR^i_d(X,\C)$. 

\subsection{Depth one resonance varieties}
\label{subsec:res one} 

Most important for us are the depth-$1$ resonance varieties, 
$\RR^i_1(X)$, and their unions up to a fixed degree, 
$\RR^i(X)=\bigcup_{j=0}^{i} \RR^j_1(X)$. The latter 
varieties can be written as
\begin{equation}
\label{eq:depth1 res}
\RR^i(X)=\{a \in A^1 \mid 
 H^j(A,\cdot a) \ne 0, \text{ for some $j\le i$}\}. 
\end{equation}
These sets provide an ascending filtration of the 
first cohomology group, 
\begin{equation}
\label{eq:filt upres}
\{0\}=\RR^0(X) \subseteq \RR^1(X) \subseteq  \cdots \subseteq 
\RR^k(X) \subseteq H^1(X,\C)=\C^n.
\end{equation}

For low values of $n=b_1(X)$, the  
variety $\RR^1(X)$ is easy to describe. 

\begin{prop}
\label{prop:low betti}
If $n\le 1$, then $\RR^1(X)=\{0\}$.
If $n=2$, then $\RR^1(X)=\C^2$ or $\{0\}$, according 
to whether the cup product vanishes on $H^1(X,\C)$ or not. 
\end{prop}

For $n\ge 3$, the resonance variety $\RR^1(X)$ can 
be much more complicated; in particular, it may have 
irreducible components which are not linear subspaces.  
The following example (a particular case of a more 
general construction described in \cite{DPS-duke}) 
illustrates this phenomenon. 

\begin{example}
\label{ex:conf spaces}
Let $X=F(T^2,3)$ be the configuration space of $3$ 
labeled points on the torus. The cohomology ring of 
$X$ is the exterior algebra on generators 
$a_1,a_2,a_3, b_1,b_2,b_3$ in degree $1$, modulo 
the ideal $\langle (a_1-a_2)(b_1-b_2),\, 
(a_1-a_3)(b_1-b_3),\, (a_2-a_3)(b_2-b_3)\rangle$. 
A calculation reveals that 
\[
\RR^1(X)=\{ (a,b) \in \C^6 \mid 
a_1+a_2+a_3=b_1+b_2+b_3=a_1b_2-a_2b_1=0\}.
\]
Hence, $\RR^1(X)$ is isomorphic to $Q=V(a_1b_2-a_2b_1)$, 
a smooth quadric hypersurface in $\C^4$.  (The projectivization 
of $Q$ is the image of the Segr\'{e} embedding of 
$\CP^1\times \CP^1$ in $\CP^3$.)  
\end{example}

The depth-$1$ resonance varieties of a product or a wedge 
of two spaces can be expressed in terms of the resonance 
varieties of the factors. Start with a product $X=X_1\times X_2$, 
where both $X_1$ and $X_2$ have finite $k$-skeleton, and identify 
$H^1(X,\C)=H^1(X_1,\C)\times H^1(X_2,\C)$. 

\begin{prop}[\cite{PS-plms}]
\label{prop:res prod}
For all $i\le k$,
\[
\RR^i_1(X_1\times X_2)= 
\bigcup_{p+q=i} \RR^{p}_1(X_1) \times \RR^{q}_1(X_2).
\]
\end{prop}

\begin{corollary}
\label{cor:rri prod}
We have 
$\RR^i(X_1\times X_2)=\bigcup_{p+q=i} \RR^{p} (X_1) 
\times \RR^{q} (X_2)$, for all $i\le k$. 
\end{corollary}

\begin{example}
\label{ex:prod surf}
For a Riemann surface of genus $g>1$, formula 
\eqref{eq:res surf} yields $\RR^0(S_g)=\{0\}$ and 
$\RR^i(S_g)=\C^{2g}$, for all $i\ge 1$.  For a product 
of two such surfaces, Corollary \ref{cor:rri prod} now gives
\begin{equation}
\label{eq:res prod surf}
\RR^i(S_{g}\times S_{h})=\begin{cases}
\{0\}  & \text{if $i=0$},\\
\C^{2g}\times \{0\} \cup \{0\}\times \C^{2h} & \text{if $i=1$},\\
\C^{2(g+h)} & \text{if $i\ge 2$}.
\end{cases}
\end{equation}
\end{example}

Next, consider a wedge $X=X_1\vee X_2$, and identify 
$H^1(X,\C)=H^1(X_1,\C)\times H^1(X_2,\C)$. 

\begin{prop}[\cite{PS-plms}]
\label{prop:res wedge}
Suppose $X_1$ and $X_2$ have positive 
first Betti numbers.  Then
\[
\RR^i_1(X) = 
\begin{cases}
H^1(X,\C) 
& \text{if $i=1$,}\\[2pt]
\RR^i_1(X_1)\times H^1(X_2,\C) 
\cup  H^1(X_1,\C)  \times \RR^i_1(X_2)
&\text{if $1<i\le k$.}
\end{cases}
\]
\end{prop}

\begin{corollary}
\label{cor:rri wedge}
Let $X=X_1\vee X_2$, where $X_1$ and $X_2$ 
have positive first Betti numbers. Then $\RR^i(X)= 
H^1(X,\C)$, for all $1\le i\le k$. 
\end{corollary}

\section{Tangent cones to affine varieties}
\label{sec:tcones vars}

We now discuss two versions of the tangent 
cone to a subvariety of the complex algebraic torus 
$(\C^{\times})^n$. 

\subsection{The tangent cone}
\label{subsec:tcone}

We start by reviewing a well-known notion in algebraic 
geometry (see \cite[pp.~251--256]{Har}).  
Let $W\subset (\C^{\times})^n$ be a Zariski closed subset, 
defined by an ideal $I$ in the Laurent polynomial ring   
$\Lambda=\C[t_1^{\pm 1},\dots , t_n^{\pm 1}]$.   
Picking a finite generating set for $I$, and multiplying 
these generators with suitable monomials if necessary, 
we see that $W$ may also be defined by the ideal $I\cap R$ 
in the polynomial ring $R=\C[t_1,\dots,t_n]$.  

Now consider the polynomial ring $S=\C[z_1,\dots, z_n]$, 
and let $J$ be the ideal generated by the polynomials 
\begin{equation}
\label{eq:fg}
g(z_1,\dots, z_n)=f(z_1+1, \dots , z_n+1),
\end{equation}
for all $f\in I\cap R$.  Finally, let $\init (J)$ be the ideal in $S$ 
generated by the initial forms of all non-zero elements from $J$.  

\begin{definition}
\label{def:tcone}
The {\em tangent cone}\/ of $W$ at $1$ is the algebraic 
subset $\TC_1(W)\subset \C^n$ defined by the initial 
ideal $\init(J)\subset S$.
\end{definition}

The tangent cone $\TC_1(W)$ is a homogeneous 
subvariety of $\C^n$, depending only on the analytic 
germ of $W$ at the identity.  In particular, 
$\TC_1(W)\ne \emptyset$ if and only if $1\in W$.  
Moreover, $\TC_1$ commutes with finite unions, 
but not necessarily with intersections.  Explicit equations for 
the tangent cone to a variety can be found using the 
Gr\"{o}bner basis algorithm described in \cite[Proposition 15.28]{Ei}.

\subsection{The exponential tangent cone}
\label{subsec:exptcone}

A competing notion of tangent cone was introduced by 
Dimca, Papadima and Suciu in \cite{DPS-duke}, and 
further studied in \cite{PS-plms} and \cite{Su}.  

\begin{definition}
\label{def:exp tcone}
The {\em exponential tangent cone}\/ of $W$ at $1$ is 
the homogeneous subvariety $\tau_1(W)$ of $\C^n$, 
defined by 
\begin{equation*}
\label{eq:tau1}
\tau_1(W)= \{ z\in \C^n \mid \exp(\lambda z)\in W,\ 
\text{for all $\lambda\in \C$} \}.
\end{equation*}
\end{definition}

The exponential tangent cone $\tau_1(W)$ depends 
only on the analytic germ of $W$ at the identity.  
In particular, $\tau_1(W)\ne \emptyset$ if and only if $1\in W$.  
Moreover, $\tau_1$ commutes with finite unions, as  
well as arbitrary intersections.  The most important 
property of this construction is given in the following 
result from \cite{DPS-duke} (see \cite{Su} for full details).

\begin{theorem}[\cite{DPS-duke}]
\label{thm:tau1}
The exponential tangent cone $\tau_1(W)$ is a finite union 
of rationally defined linear subspaces of $\C^n$.
\end{theorem}

These subspaces can be described explicitly.  
By the above remarks, we may assume $W=V(f)$, where  
$f$ is a non-zero Laurent polynomial with $f(1)=0$. 
Write $f= \sum_{a\in S} c_a t_1^{a_1}\cdots t_n^{a_n}$, 
where $S$ is a finite subset of $\Z^n$, and $c_a\in \C\setminus \{0\}$ 
for each $a=(a_1,\dots,a_n)\in S$. We say a partition 
$\cS=(\cS_1 \, | \cdots | \, \cS_q)$ of the support $S$ is 
{\em admissible}\/ if $\sum_{a\in \cS_i} c_a =0$, for all 
$1\le i\le q$. For each such partition, let $L(\cS)$ 
be the rational linear subspace consisting of all points 
$x\in \Q^n$ for which the dot product $(a-b)\cdot x$ 
vanishes, for all $a,b\in \cS_i$ and all $1\le i\le q$. 
Then
\begin{equation}
\label{eq:tau1w}
\tau_1(W)= \bigcup_{\cS} L(\cS)\otimes \C,
\end{equation}
where the union is taken over all (maximal) admissible 
partitions of $S$.

\subsection{Relating the two tangent cones}
\label{subsec:tc exptc}

The next lemma (first noted in \cite{DPS-duke}) records 
a general relationship between the two kinds of tangent 
cones discussed here.  For completeness, we include 
a proof, along the lines of the proof given in \cite{PS-plms}, 
though slightly modified to fit the definition of $\TC_1$ 
given here.

\begin{lemma}[\cite{DPS-duke, PS-plms}] 
\label{lem:ttcone}
For every Zariski closed subset $W\subset (\C^{\times})^n$, 
we have $\tau_1(W)\subseteq \TC_1(W)$. 
\end{lemma}

\begin{proof}
Without loss of generality, we may assume $1\in W$. 
Let $f\in R$ be a non-zero polynomial in $I=I(W)$, 
let $g\in S$ be the polynomial defined by \eqref{eq:fg}, 
and let $g_0=\init(g)$. We then have
\begin{align*}
f(e^{\lambda z_1},\dots, e^{\lambda z_n})
&=f(1+\lambda z_1+O(\lambda^2) ,\dots, 1+\lambda z_n+O(\lambda^2)  )\\
&=g(\lambda z_1+O(\lambda^2)  ,\dots,\lambda z_n +O(\lambda^2) )\\
&=g_0(\lambda z_1 ,\dots,\lambda z_n  )+\text{higher order terms}.
\end{align*}

Now suppose $z\in \tau_1(W)$, that it to say, 
$f(e^{\lambda z_1},\dots, e^{\lambda z_n})=0$, 
for all $\lambda\in \C$. The above calculation  
shows that $g_0(\lambda z)=0$, for all $\lambda$;  
in particular, $g_0(z)=0$.  We conclude that $z\in \TC_1(W)$. 
\end{proof}

If $W$ is an algebraic subtorus of $(\C^{\times})^n$, then 
$\tau_1(W)=\TC_1(W)$, and both coincide with the tangent 
space at the origin, $T_1(W)$.    In general, though, 
the inclusion from Lemma \ref{lem:ttcone} can be strict, 
even when $1$ is a smooth point of $W$.  Here is an 
example illustrating this phenomenon. 

\begin{example}
\label{ex:chain link} 
Let $W$ be the  hypersurface in $(\C^{\times})^3$ 
with equation $ t_1+t_2+t_3-t_1t_2-t_1t_3-t_2t_3=0$.  Then 
$\tau_1(W)$ is a union of $3$ lines in $\C^3$, given by 
the equations $ z_1=z_2+z_3=0$,  $z_2=z_1+z_3=0$,  
and $z_3=z_1+z_2=0$. On the other hand, 
$\TC_1(W)$ is a plane in $\C^3$, 
with equation $z_1+z_2+z_3=0$.  Hence, $\TC_1(W)$ 
strictly contains $\tau_1(W)$. 
\end{example}

\section{Characteristic varieties}
\label{sec:cvs}

In this section, we briefly review the characteristic 
varieties of a space, and their relation to the 
resonance varieties. 

\subsection{Jump loci for twisted homology}
\label{subsec:jumps} 
Let $X$ be a connected CW-complex with finite $k$-skeleton, 
$k\ge 1$.  Without loss of generality, we may assume $X$ 
has a single $0$-cell, call it $x_0$.  Let $G=\pi_1(X,x_0)$ 
be the fundamental group of $X$, and let $\wG=\Hom(G,\C^{\times})$ 
be the group of complex characters of $G$.  Clearly, 
$\wG=\wG_{\ab}$, where $G_{\ab}=H_1(X,\Z)$ is the 
abelianization of $G$.  Thus, the universal coefficient 
theorem allows us to identify 
\begin{equation}
\label{eq:wg}
\wG = H^1(X,\C^{\times}).
\end{equation}

Each character $\rho\colon G\to \C^{\times}$ determines 
a rank~$1$ local system $\L_{\rho}$ on our space $X$.  
Computing the homology groups of $X$ with coefficients 
in such local systems leads to a natural filtration of the 
character group. 

\begin{definition}
\label{def:cvs}
The {\em characteristic varieties}\/ of $X$ are the sets 
\begin{equation*}
\label{eq:cvs}
\VV^i_d(X)=\set{\rho \in H^1(X,\C^{\times})  
\mid \dim_{\C} H_i(X,\L_{\rho})\ge d},
\end{equation*}
defined for all $0\le i\le k$ and all $d>0$. 
\end{definition}

Clearly, $1\in \VV^i_d(X)$ if and only if $d\le b_i(X)$. 
In degree $0$, we have $\VV^0_1(X)= \{ 1\}$ 
and $\VV^0_d (X)=\emptyset$, for $d>1$. 
In degree $1$, the sets $\VV^i_d(X)$ depend only 
on the group $G=\pi_1(X,x_0)$---in fact, only on its 
maximal metabelian quotient, $G/G''$.

The jump loci $\VV^i_d(X)$ are Zariski closed subsets 
of the algebraic group $H^1(X,\C^{\times})$; moreover, these 
varieties are homotopy-type invariants of $X$.  For details 
and further references, see \cite{Su}.

\subsection{Depth one characteristic varieties}
\label{subsec:v1} 
Most important for us are the depth one characteristic 
varieties, $\VV^i_1(X)$, and their unions up to a fixed 
degree, $\VV^i(X)=\bigcup_{j=0}^{i} \VV^j_1(X)$. 
These varieties yield an ascending filtration of 
the character group, 
\begin{equation}
\label{eq:filt asc}
\{1\}=\VV^0(X) \subseteq \VV^1(X) \subseteq  \cdots 
\subseteq \VV^k(X) \subseteq H^1(X,\C^{\times}).
\end{equation}

Let $\wG^0=H^1(X,\C^{\times})^0$ be the identity 
component of the character group $\wG$. Writing 
$n=b_1(X)$, we may identify $\wG^0$ with the 
complex algebraic torus $(\C^{\times})^n$. 
Set $\WW^i_d(X)=\VV^i_d(X)\cap \wG^0$, and 
$\WW^i(X)=\bigcup_{j=0}^{i} \WW^j_1(X)$.  We 
then have 
\begin{equation}
\label{eq:wi}
\WW^i(X)=\VV^i(X)\cap \wG^0.  
\end{equation}

\begin{example}
\label{ex:cv link}
Let $L$ be an $n$-component link in $S^3$, with complement 
$X$.   Using a basis for $H_1(X,\Z)=\Z^n$ given by 
(oriented) meridians, we may identify 
$H^1(X,\C^{\times})=(\C^{\times})^n$.  Then 
\begin{equation}
\label{eq:cv link}
\WW^1(X)= 
\{z\in (\C^{\times})^n \mid \Delta_L(z)=0\} \cup \{1\},
\end{equation}
where $\Delta_L=\Delta_L(t_1,\dots ,t_n)$ is 
the Alexander polynomial of the link. 
For details and references, see \cite{S09}.
\end{example}

As shown in \cite{PS-plms} (see also \cite{Su}), the 
characteristic varieties satisfy product and wedge 
formulas similar to those satisfied by the resonance 
varieties. 

\begin{prop}
\label{prop:cv pw}
Let $X_1$ and $X_2$ be connected CW-complexes 
with finite $k$-skeleton, and fix an integer  $1\le i\le k$. 
Then
\begin{enumerate}
\item \label{cp1}
$\WW^i(X_1\times X_2)= 
\bigcup_{p+q=i} \WW^{p}(X_1) \times \WW^{q}(X_2)$.
\item \label{cp2}
Suppose, moreover, that $b_1(X_1)>0$ and $b_1(X_2)>0$. 
Then $\WW^i(X_1\vee X_2)= H^1(X_1\vee X_2,\C^{\times})^0$.
\end{enumerate}
\end{prop}

\subsection{Characteristic subspace arrangements}
\label{subsec:df subarr}

As before, let $X$ be a connected CW-complex with 
finite $k$-skeleton.  Set $n=b_1(X)$, and identify 
$H^1(X, \C)=\C^n$ and $H^1(X, \C^{\times})^0=(\C^{\times})^n$.   
Applying Theorem  \ref{thm:tau1} to the characteristic 
varieties $\WW^i(X)\subseteq (\C^{\times})^n$
leads to the following definition.  

\begin{definition}
\label{def:df subarr}
For each $i\le k$, the {\em $i$-th characteristic arrangement}\/ 
of $X$, denoted $\CC_i(X)$, is the subspace arrangement in 
$H^1(X,\Q)$ whose complexified union is the exponential 
tangent cone to $\WW^i(X)$:
\begin{equation*}
\label{eq:dfxi}
\tau_1( \WW^i(X)) =  
\bigcup_{L\in \CC_i(X)} L\otimes \C.
\end{equation*}
\end{definition}

We thus have a sequence $\CC_0(X)$, $\CC_1(X), \dots , 
\CC_k(X)$ of rational subspace arrangements, all lying in 
the same affine space $H^1(X,\Q)=\Q^n$.  As noted in 
\cite[Lemma 7.7]{Su}, these subspace arrangements 
depend only on the homotopy type of $X$.  

\subsection{Tangent cone and resonance}
\label{subsec:tcone res}

Of great importance in the theory of cohomology jumping 
loci is the relationship between characteristic and 
resonance varieties, based on the tangent cone 
construction. A foundational result in this direction 
is the following theorem of Libgober \cite{Li02}.

\begin{theorem}[\cite{Li02}]
\label{thm:lib}
The tangent cone at $1$ to $\WW_d^i(X)$ is included 
in $\RR_d^i(X)$, for all $i\le k$ and $d>0$.  
\end{theorem}
 
Since tangent cones commute with finite unions, 
we get: 
\begin{equation}
\label{eq:lib}
\TC_1(\WW^i(X))\subseteq \RR^i(X), \quad
\text{for all $i\le k$}.
\end{equation}

For many spaces of interest, inclusion \eqref{eq:lib} 
holds as an equality.  In general, though, the inclusion is strict. 

\begin{example}
\label{ex:tc heis}
Let $M$ be the $3$-dimensional Heisenberg nilmanifold.   
It is easily seen that $\WW^1(M)=\{1\}$; hence, 
$\TC_1(\WW^1(M))=\{0\}$.  On the other hand, 
$\RR^1(M)=\C^2$, since the cup product vanishes on 
$H^1(M,\C)$. 
\end{example}

\section{Straight spaces}
\label{sec:straight}

We now delineate a class of spaces for which 
the characteristic and resonance varieties have a 
simple nature, and are intimately related to each other 
via the tangent cone constructions discussed 
in \S\ref{sec:tcones vars}.

\subsection{Straightness}
\label{subsec:straight}

As before, let $X$ be a connected CW-complex with finite 
$k$-skeleton.  For each $i\le k$, consider the following 
conditions on the varieties $\WW^i(X)$ and $\RR^i(X)$:

\begin{alphenum}
\item \label{v1}
All components of $\WW^{i}(X)$ passing through the origin 
are algebraic subtori.
\item \label{v2}
$\TC_1(\WW^{i}(X))=\RR^{i}(X)$. 
\item \label{v3}
All components of $\WW^{i}(X)$ not passing through the 
origin are $0$-dimensional.
\end{alphenum}

\begin{definition}
\label{def:loc straight}
We say $X$ is {\em locally $k$-straight}\/ if  
conditions \eqref{v1} and \eqref{v2} hold, for each $i\le k$.  
If these conditions hold for all $k\ge 1$, we say $X$ is a 
{\em locally straight space}.
\end{definition}

\begin{definition}
\label{def:straight}
We say $X$ is {\em $k$-straight}\/ if conditions 
\eqref{v1}, \eqref{v2}, and \eqref{v3} hold, for each $i\le k$.  
If these conditions hold for all $k\ge 1$, we say $X$ is a 
{\em straight space}.
\end{definition}

Clearly, the $k$-straightness property depends only 
on the homotopy type of a given space.  In view of this observation, 
we may declare a group $G$ to be $k$-straight if there is 
a classifying space $K(G,1)$ which is $k$-straight; 
in particular, such a group $G$ must be of type $\FF_k$, 
i.e., have a $K(G,1)$ with finite $k$-skeleton. 
Note that a space is $1$-straight if and only if 
its fundamental group is $1$-straight.  Similar 
considerations apply to local straightness.

The straightness conditions are quite stringent, and thus 
easily violated. Here are a few examples when this happens, 
all for $k=1$. 

\begin{example}
\label{ex:straight a}
The closed three-link chain $L$ (the link $6^3_1$ 
from Rolfsen's tables) has Alexander polynomial
$\Delta_L=t_1+t_2+t_3-t_1t_2-t_1t_3-t_2t_3$.  
Let $X=S^3\setminus L$ be the link complement. 
By \eqref{eq:cv link}, the characteristic variety 
$W=\WW^1(X)$ has equation $\Delta_L=0$. 
On the other hand, we know 
from Example \ref{ex:chain link}  that 
$\tau_1(W)\ne \TC_1(W)$. Thus, $W$ is not 
an algebraic torus, and so $X$ is not locally 
$1$-straight:  condition \eqref{v1} fails.  
\end{example}

\begin{example} 
\label{ex:straight b}
The Heisenberg nilmanifold $M$ 
from Example \ref{ex:tc heis} is not locally $1$-straight:  
condition \eqref{v1} is met, but not condition \eqref{v2}.
\end{example}

\begin{example}
\label{ex:straight c}
Let $G=\langle x_1, x_2 \mid x_1^2 x_2 = x_2 x_1^2 \rangle$ 
be the group from \cite[Example 6.4]{PS-plms}. 
Then $\WW^1(G)=\{1\}\cup \{(t_1,t_2)\in (\C^{\times})^2 \mid t_1=-1\}$, 
and $\RR^1(G)=\{0\}$.  Thus, $G$ is locally $1$-straight, 
but not $1$-straight:  conditions \eqref{v1} and \eqref{v2} 
are met, but not condition \eqref{v3}.
\end{example}

Nevertheless, as we will see below, there is an abundance 
of interesting spaces which are, to a degree on another, straight. 

\subsection{Examples and discussion}
\label{subsec:straight examples}
  
To start with, let us consider the case when the first 
characteristic variety $\WW^1(X)$ is as big as it can be. 

\begin{lemma}
\label{lem:straight}
Let $X$ be a finite-type CW-complex.
Suppose that $\WW^1(X)=H^1(X,\C^{\times})^{0}$.  
Then $X$ is straight. 
\end{lemma}

\begin{proof}
Our assumption forces $\WW^i(X)=H^1(X,\C^{\times})^{0}$, 
for all $i\ge 1$.  Straightness conditions \eqref{v1} 
and \eqref{v3} are automatically met, while 
condition \eqref{v2}  follows from \eqref{eq:lib}. 
\end{proof}

Examples of spaces satisfying the hypothesis of the above 
lemma include the Riemann surfaces $S_g$ with $g\ge 2$.

\begin{lemma}
\label{lem:straight crit}
Let $X$ be a CW-complex with finite $k$-skeleton, $k\ge 1$.
Suppose that $\WW^k(X)$ is finite, and $\RR^k(X)=\{0\}$.  
Then $X$ is $k$-straight. 
\end{lemma}

\begin{proof}
Since $\WW^k(X)$ is finite, conditions \eqref{v1} and \eqref{v3} 
are trivially satisfied. By formula \eqref{eq:lib}, we have 
$\TC_1(\WW^i(X))\subseteq \RR^i(X)=\{0\}$, for 
all $i\le k$.  Since $1\in \WW^i(X)$, equality must 
hold, i.e., condition \eqref{v2} is satisfied. 
\end{proof}

Examples of spaces satisfying the hypothesis of the above 
lemma include the tori $T^n$; these spaces are, in fact, straight.

Evidently, all spaces $X$ with $b_1(X)= 0$ are straight. 
Next, we deal with the case $b_1(X)=1$. The following 
example shows that not all finite CW-complexes with first 
Betti number $1$ are (locally) straight.

\begin{example}
\label{ex:nonstraight}
Let $f$ be a polynomial in $\Z[t]$ with $f(1)=0$, and let 
$X_f=(S^1\vee S^2)\cup_{\varphi} e^3$ be the corresponding 
minimal CW-complex constructed in Example \ref{ex:s1s2}.  
A calculation with the equivariant 
chain complex \eqref{eq:s1s2ch} shows that 
$\WW^1(X_f)=\{1\}$ and $\WW^2(X_f)=V(f)$; clearly, both 
these are sets finite subsets of $H^1(X,\C^{\times})=\C^{\times}$. 

An analogous calculation with the universal Aomoto cochain 
complex \eqref{eq:s1s2aom} shows that $\RR^1(X_f)=\{0\}$, 
yet $\RR^2(X_f)=\{0\}$ only if $f'(1)\ne 0$, but 
$\RR^2(X_f)=\C$, otherwise.
Therefore, $X_f$ 
is always $1$-straight, but 
\begin{equation}
\label{eq:2str}
\text{$X_f$ is locally $2$-straight $\same$ 
$f'(1)\ne 0$.}
\end{equation}
\end{example}

Applying a similar reasoning to cell complexes of the 
form $X_f=(S^1\vee S^k)\cup_{\varphi} e^{k+1}$, with 
attaching maps corresponding to polynomials 
$f\in \Z[t]$ with $f(1)=f'(1)=0$, we obtain the 
following result.

\begin{prop}
\label{prop:sns}
For each $k\ge 2$, there is a minimal CW-complex  
which has the integral homology of $S^1\times S^k$ 
and which is $(k-1)$-straight, but not locally $k$-straight.
\end{prop}

Nevertheless, we have the following positive result, 
guaranteeing straightness in a certain range for 
spaces with first Betti number $1$.

\begin{prop}
\label{prop:b1 straight}
Let $X$ be a CW-complex with finite 
$k$-skeleton.  Assume $b_1(X)=1$.  Then,
\begin{enumerate}
\item \label{b1-1} $X$ is $1$-straight.
\item \label{b1-2} If, moreover, $b_i(X)=0$ for $1<i\le k$, 
then $X$ is $k$-straight. 
\end{enumerate}
\end{prop}

\begin{proof}
As noted in Proposition \ref{prop:low betti}, the fact 
that $b_1(X)=1$ implies $\RR^1(X)=\{0\}$. 
Identify $H^1(X,\C^{\times})=\C^{\times}$. 
By formula \eqref{eq:lib}, we have 
$\TC_1(\WW^1(X))=\{0\}$.  Hence, 
$\WW^1(X)$ is a proper subvariety of $\C^{\times}$; 
consequently, it is a finite set.  Part \eqref{b1-1} now 
follows from Lemma \ref{lem:straight crit}.

For Part \eqref{b1-2}, note that $\RR^i(X)=\{0\}$, for all $i\le k$. 
As above, we conclude that $\WW^k(X)$ is finite, 
and hence $X$ is $k$-straight.
\end{proof}

\begin{corollary}
\label{cor:knots}
Let $K$ be a smoothly embedded $d$-sphere in $S^{d+2}$,  
and let $X=S^{d+2}\setminus K$ be its complement.  Then 
\begin{enumerate}
\item \label{knot1} $\RR^k(X)=\{0\}$ and $\WW^k(X)$ is finite, 
for all $k\ge 1$.
\item \label{knot2} $X$ is straight. 
\end{enumerate}
\end{corollary}

\begin{proof}
The knot complement has the homotopy type of a 
$(d+1)$-dimensional CW-complex, with the integral 
homology of $S^1$.  The desired conclusions 
follow from Proposition \ref{prop:b1 straight} and its proof.
\end{proof} 

For a knot $K$ in $S^3$, the variety $\WW^1(X)\subset \C^{\times}$ 
consists of $1$, together with all the roots of the Alexander 
polynomial, $\Delta_K$ (these roots are always different from $1$).

\subsection{Products and wedges}
\label{subsec:straight pw}
We now look at how the straightness properties 
behave with respect to (finite) products and wedges 
of spaces.

\begin{prop}
\label{prop:straight wedge}
Let $X_1$ and $X_2$ be CW-complexes with finite 
$k$-skeleta and positive first Betti numbers. Then 
the space $X=X_1\vee X_2$ is $k$-straight.
\end{prop}

\begin{proof}
Follows from Proposition \ref{prop:cv pw}\eqref{cp2}  
and Lemma \ref{lem:straight}.
\end{proof}

In particular, a wedge of knot complements is straight. 

\begin{prop}
\label{prop:loc straight prod}
Let $X_1$ and $X_2$ be two CW-complexes. 
\begin{enumerate}
\item \label{ls1}
If $X_1$ and $X_2$ are locally $k$-straight, then so is 
$X_1\times X_2$.
\item \label{ls2} 
If $X_1$ and $X_2$ are $1$-straight, then so is 
$X_1\times X_2$.
\end{enumerate}
\end{prop}

\begin{proof}
Both assertions follow from the product formulas for resonance 
and characteristic varieties (Proposition \ref{prop:cv pw}\eqref{cp1} and 
Corollary \ref{cor:rri prod}, respectively).
\end{proof}

\begin{prop}
\label{prop:straight prod}
Let  $X_1$ and $X_2$ be CW-complexes with finite $k$-skeleton. 
Suppose $\RR^k(X_j)=\{0\}$ and $\WW^k(X_j)$ is finite, 
for $j=1,2$.  Then $X_1\times X_2$ is $k$-straight.
\end{prop}

\begin{proof}
From Lemma \ref{lem:straight crit}, we know that both 
$X_1$ and $X_2$ are $k$-straight.  Now, 
Corollary \ref{cor:rri prod} gives that $\RR^k(X_1\times X_2)=\{0\}$, 
while Proposition \ref{prop:cv pw}\eqref{cp1} gives that 
$\WW^k(X_1\times X_2)$ is finite. Hence, again by 
 Lemma \ref{lem:straight crit}, $X_1\times X_2$ is $k$-straight. 
\end{proof}

In particular, a product of knot complements is straight.
In general, though, a product of straight spaces need 
not be straight.  

\begin{example}
\label{ex:nonstraight product}
Let $K$ be a knot in $S^3$, with Alexander polynomial 
$\Delta_K$ not equal to $1$ (for instance, the trefoil knot), and 
let $X$ be its complement. Let $Y=\bigvee^n S^1$ 
be a wedge of $n$ circles, with $n\ge 2$. 
By Corollary \ref{cor:knots} and 
Proposition \ref{prop:straight wedge}, both $X$ and $Y$ 
are straight. 

By Proposition \ref{prop:loc straight prod}\eqref{ls2}, the 
product $X\times Y$ is $1$-straight.  Nevertheless, 
$X\times Y$  is not $2$-straight. Indeed, the variety 
$\WW^2(X\times Y)$ has irreducible components of the 
form $\{\rho\} \times (\C^{\times})^n$, where $\rho$ runs 
through the roots of $\Delta_K$.  As these components 
do not pass through the origin, straightness 
condition \eqref{v3} fails for $X\times Y$ in degree $i=2$. 
\end{example} 

\subsection{Rationality properties}
\label{subsec:straight rat}

Before proceeding, let us give an alternative characterization 
of straightness.  As usual, let $X$ be a connected CW-complex 
with finite $k$-skeleton, and let 
$\exp\colon H^1(X,\C) \to H^1(X,\C^{\times})^0$ 
be the exponential map (or rather, its corestriction 
to its image).

\begin{theorem}
\label{thm:alt straight}
A space $X$ as above is locally $k$-straight if and only 
if the following equalities hold, for all $i\le k$:
\begin{align*}
\tag{$\alpha$}\label{w1}
\WW^{i}(X)&=\Bigg(\bigcup_{L\in \CC_i(X)} \exp(L\otimes \C)\Bigg) 
\cup Z_i\\
\tag{$\beta$}\label{w2}
\RR^{i}(X)&=\bigcup_{L\in \CC_i(X)} L\otimes \C,
\end{align*}
for some algebraic subsets $Z_i\subset H^1(X,\C^{\times})^0$ 
not containing the origin.  

The space $X$ is $k$-straight if and only if, in addition, the sets 
$Z_i$ are finite.
\end{theorem}

\begin{proof}
Suppose $X$ is locally $k$-straight, and fix an index $i\le k$.  
By hypothesis \eqref{v1} from Definition \ref{def:straight}, 
the variety $\WW^i(X)$ admits a decomposition into 
irreducible components of the form  
\begin{equation}
\label{eq:wstraight}
\WW^i(X) = \bigcup_{\alpha\in \Lambda} T_{\alpha}\cup 
Z_i,
\end{equation} 
for some algebraic subtori $T_{\alpha}=\exp(P_{\alpha}\otimes \C)$ 
and some algebraic sets $Z_i$ with $1\notin Z_i$.  Hence,
\begin{equation}
\label{eq:tauwstraight}
\tau_1(\WW^{i}(X))=\bigcup_{\alpha \in \Lambda} P_{\alpha}\otimes \C.
\end{equation}

On the other hand, by the definition of characteristic 
subspace arrangements, 
we also have  
$\tau_1(\WW^{i}(X))=\bigcup_{L\in \CC_i(X)} L\otimes \C$. 
By uniqueness of decomposition into irreducible components,  
the arrangement $\{P_{\alpha}\}_{\alpha\in \Lambda}$ must 
coincide with $\CC_i(X)$. Consequently, equation \eqref{eq:wstraight} 
yields \eqref{w1}. 

Equation \eqref{eq:wstraight} also implies 
$\tau_1(\WW^{i}(X))=\TC_1(\WW^{i}(X))$.  
Hypothesis \eqref{v2} then gives  
$\tau_1(\WW^{i}(X))=\RR^{i}(X)$, 
which is precisely condition \eqref{w2}.

Conversely, condition \eqref{w1} implies \eqref{v1} and 
conditions  \eqref{w1} and \eqref{w2} together imply 
\eqref{v2}.  

Finally, hypothesis \eqref{v3} is satisfied if and only 
if the sets $Z_1, \dots ,Z_k$ are all finite.
\end{proof}


\begin{corollary}
\label{cor:rat res}
Let $X$ be a locally $k$-straight space.  Then, for all $i\le k$,
\begin{enumerate}
\item \label{rs1}
$\tau_1(\WW^i(X))=\TC_1(\WW^i(X))=\RR^i(X)$.
\item \label{rs2}
$\RR^i(X,\Q)=
\bigcup_{L\in \CC_i(X)} L$. 
\end{enumerate}
In particular, the resonance varieties $\RR^i(X)$ 
are unions of rationally defined subspaces. 
\end{corollary}

The next example (adapted from \cite{DPS-duke}) 
illustrates how this rationality property may be used to 
detect non-straightness. 

\begin{example}
\label{ex:sqr2}
Consider the group $G$ with generators $x_1, x_2, x_3, x_4$ 
and relators $r_1=[x_1, x_2]$, $r_2=[x_1, x_4] [x_2^{-2}, x_3]$, 
$r_3= [x_1^{-1}, x_3] [x_2, x_4]$.  
Direct computation shows that
\[
\RR^1(G)=\{z \in \C^4 \mid z_1^2-2z_2^2=0\}. 
\] 
Evidently, this variety splits into two linear subspaces defined 
over $\R$, but not over $\Q$.  Thus, $G$ is not (locally) 
$1$-straight. 
\end{example}

\section{The Dwyer--Fried invariants}
\label{sec:df res}

In this section, we recall the definition of the Dwyer--Fried 
sets, and the way these sets relate to the characteristic varieties 
of a space.

\subsection{Betti numbers of free abelian covers}
\label{subsec:df}

As before, let $X$ be a connected CW-complex 
with finite $k$-skeleton, and let $G=\pi_1(X,x_0)$.  
Denote by $n=b_1(X)$ the first Betti number of 
$X$.  (We may as well assume $n>0$, otherwise the 
whole theory is empty of content.) Fix an integer $r$ 
between $1$ and $n$, and consider the regular covers 
of $X$, with group of deck-transformations $\Z^r$.  

Each such cover, $X^{\nu}\to X$, is determined 
by an epimorphism $\nu \colon G \surj \Z^r$.  The 
induced homomorphism in cohomology, 
$\nu^*\colon H^1(\Z^r,\Q) \inj H^1(G,\Q)$, defines  
an $r$-dimensional subspace, $P_{\nu}=\im(\nu^*)$,  
in the rational vector space $H^1(G,\Q)=\Q^n$.  
Conversely, each $r$-dimensional subspace 
$P\subset \Q^n$ can be written as $P=P_{\nu}$, 
for some epimorphism $\nu \colon G \surj \Z^r$, 
and thus defines a regular $\Z^r$-cover of $X$. 

To recap, the regular $\Z^r$-covers of $X$ are 
parametrized by the Grassmannian of $r$-planes in 
$H^1(X,\Q)$, via the correspondence 
\[
\big\{ \text{$\Z^r$-covers $X^{\nu}  \to X$} \big\} \longleftrightarrow 
\big\{ \text{$r$-planes $P_{\nu}:=\im(\nu^*)$ in $H^1(X,\Q)$}\big\}.
\]

Moving about the rational Grassmannian 
and recording how the Betti numbers of the corresponding 
covers vary leads to the following definition.

\begin{definition}
\label{def:df}
The {\em Dwyer--Fried invariants}\/ of $X$ are 
the subsets 
\begin{equation}
\label{eq:grass}
\Omega^i_r(X)=\big\{P_{\nu} \in \Grass_r(H^1(X,\Q)) \bigmid 
\text{$b_{j} (X^{\nu}) <\infty$ for $j\le i$}  \big\}. 
\end{equation}
\end{definition}

Set $n=b_1(X)$.  For a fixed $r$ between $1$ and $n$, 
these sets form a descending filtration of the Grassmannian 
of $r$-planes in $H^1(X,\Q)=\Q^n$, 
\begin{equation}
\label{eq:df filt}
\Grass_r(\Q^n) = \Omega^0_r(X) \supseteq \Omega^1_r(X)  
\supseteq \Omega^2_r(X)  \supseteq \cdots.
\end{equation}
If $r>n$, we adopt the convention that $\Grass_r(\Q^n)=\emptyset$ 
and define $\Omega^i_r(X)=\emptyset$ in this range.

As noted in \cite{Su}, the $\Omega$-sets are homotopy-type 
invariants.  More precisely, if  $f\colon X\to Y$ is a homotopy 
equivalence, the induced isomorphism in cohomology, 
$f^*\colon H^1(Y,\Q) \to H^1(X,\Q)$, defines isomorphisms 
$f^*_r\colon \Grass_r(H^1(Y,\Q)) \to \Grass_r(H^1(X,\Q))$, 
which send each subset $\Omega^i_r(Y)$ bijectively 
onto $\Omega^i_r(X)$.  

\begin{example}
\label{ex:torus}
Let $T^n$ be the $n$-dimensional torus. Since every 
connected cover of $T^n$ is homotopy equivalent to a 
$k$-torus, for some $0\le k\le n$, we conclude that 
$\Omega^i_r(T^n)=\Grass_r(\Q^n)$, for all $i\ge 0$  
and $r\ge 1$.
\end{example}

\subsection{Dwyer--Fried invariants and characteristic varieties}
\label{subsec:df cv}

The next theorem reduces the computation of the $\Omega$-sets 
to a more standard computation in algebraic geometry.  
The theorem was proved by Dwyer and Fried in \cite{DF}, 
using the support loci for the Alexander invariants, 
and was recast in a slightly more general context by 
Papadima and Suciu in \cite{PS-plms}, using the 
characteristic varieties.  We state this result in the  
form most convenient for our purposes, namely, the 
one established in \cite{Su}.

\begin{theorem}[\cite{DF, PS-plms, Su}]
\label{thm:df cv}
For all $i\le k$ and $1\le r\le n$, 
\begin{equation*}
\label{eq:grass cv}
\Omega^i_r(X)=\big\{P\in \Grass_r(\Q^n) \bigmid
\# \big(\!\exp(P \otimes \C) \cap \WW^i(X) 
\big)<\infty \big\}.
\end{equation*}
\end{theorem}

In other words, an $r$-plane $P\subset \Q^n$ belongs 
to $\Omega^i_r(X)$ if and only if the algebraic torus 
$T=\exp(P \otimes \C)$ intersects the characteristic 
variety $W=\WW^i(X)$ only in finitely many points.  When 
this happens, the exponential tangent cone $\tau_1(T\cap W)$ 
equals $\{0\}$, forcing $P\cap L=\{0\}$, for 
every subspace $L\subset \Q^n$ in the characteristic 
subspace arrangement $\CC_i(X)$.   As in 
\cite[Theorem 8.1]{Su}, we obtain the following 
``upper bound" for the Dwyer--Fried invariants of $X$:
\begin{equation}
\label{eq:ubound}
\Omega^i_r(X) \subseteq \bigg(
\bigcup_{L\in \CC_i(X)} \big\{P \in \Grass_r(H^1(X,\Q)) \bigmid 
P\cap L \ne \{0\} \big\} \bigg)^{\compl}.
\end{equation}

\subsection{The incidence correspondence}
\label{subsec:incidence}

The right side of \eqref{eq:ubound} may be reinterpreted 
in terms of the classical incidence correspondence from 
algebraic geometry. 

Let $V$ be a homogeneous variety in $\k^n$.  
Consider  the locus of $r$-planes in $\k^n$ meeting $V$, 
\begin{equation}
\label{eq:incident}
\sigma_r(V) = \big\{ P \in \Grass_r(\k^n) 
\bigmid P \cap  V \ne \{0\} \big\}.
\end{equation}
This set is a Zariski closed subset of the Grassmannian 
$\Grass_r(\k^n)$, called the {\em variety of incident 
$r$-planes}\/ to $V$. 

Particularly manageable is the case when $V$ is a 
non-zero linear subspace $L\subset \k^n$.  The corresponding 
incidence variety, $\sigma_r(L)$, is known as the  
{\em special Schubert variety}\/ defined by $L$. 
If $L$ has codimension $d$ in $\k^n$, then 
$\sigma_r(L)$ has codimension $d-r+1$ in $\Grass_r(\k^n)$. 

\begin{theorem}[\cite{Su}]
\label{thm:sch bound}
Let $X$ be a CW-complex with finite $k$-skeleton. 
Then
\begin{equation*}
\label{eq:omega schubert}
\Omega^i_r(X) \subseteq  \Grass_r(H^1(X,\Q)) \setminus 
\bigcup_{L\in \CC_i(X)} \sigma_r(L ),
\end{equation*}
for all $i\le k$ and $r\ge 1$. 
\end{theorem}

In other words, each Dwyer--Fried set $\Omega^i_r(X)$ 
is contained in the complement to the variety 
of incident $r$-planes to the $i$-th characteristic 
arrangement of $X$. 

\subsection{Straightness and the Dwyer--Fried invariants}
\label{subsec:df res vars}

Under a (local) straightness assumption, the bound 
from Theorem \ref{thm:sch bound} can be expressed 
in terms of simpler, purely cohomological data.  The 
next result proves Theorem \ref{thm:intro main}, part 
\eqref{main1} from the Introduction.

\begin{corollary}
\label{cor:res bound}
Suppose $X$ is locally $k$-straight. Then, for all $i\le k$ 
and $r\ge 1$, 
\[
\Omega^i_r(X) \subseteq \Grass_r(H^1(X,\Q)) \setminus 
\sigma_r(\RR^i(X,\Q)).
\]
\end{corollary}

\begin{proof}
By Corollary \ref{cor:rat res}, we have 
$\tau_1(\WW^i(X))=\RR^i(X)$. 
Hence, $\bigcup_{L\in \CC_i(X)} L= \RR^i(X,\Q)$, 
and the desired conclusion follows from 
Theorem \ref{thm:sch bound}.
\end{proof}

Under a more stringent straightness assumption, 
the above inclusion holds as an equality.   The 
next result proves Theorem \ref{thm:intro main}, 
part \eqref{main2} from the Introduction.

\begin{theorem}
\label{thm:straight omega}
Suppose $X$ is $k$-straight. Then, for all $i\le k$ 
and $r\ge 1$,
\[
\Omega^i_r(X) = \Grass_r(H^1(X,\Q)) \setminus 
\sigma_r(\RR^i(X,\Q)).
\]
\end{theorem}

\begin{proof}
By Theorem \ref{thm:df cv}, we have
\begin{equation}
\label{eq:omega ww}
\Omega^i_r(X)  = 
\big\{P \bigmid
\exp(P \otimes \C) \cap \WW^i(X)\  
\text{is finite} \big\}.
\end{equation}
Since $X$ is $k$-straight, Theorem \ref{thm:alt straight}, 
part \eqref{w1} yields
\begin{equation}
\label{eq:omega cc}
\Omega^i_r(X)  
=\big\{P \bigmid
P\cap L =\{0\},\, \text{for all $L\in \CC_i(X)$} \big\}. 
\end{equation}
By Theorem \ref{thm:alt straight}, 
part \eqref{w2}, the right side of \eqref{eq:omega cc} 
is the complement to $\sigma_r(\RR^i(X,\Q))$, and we 
are done.
\end{proof}

Particularly interesting is the case when all the components 
of $\RR^i(X)$ have the same codimension, say, $r$. 
In this situation, $\Omega^i_r(X)$ is the complement 
of the rational Chow divisor of $\RR^i(X,\Q)$. 

The previous theorem yields a noteworthy qualitative result about 
the Dwyer--Fried sets of straight spaces, in arbitrary ranks $r\ge 1$. 

\begin{corollary}
\label{cor:omega deep}
Let $X$ be a $k$-straight space.  Then each set $\Omega^i_r(X)$ 
is the complement of a finite union of special Schubert 
varieties in the Grassmannian of $r$-planes in $H^1(X,\Q)$.  
In particular, $\Omega^i_r(X)$ is a Zaris\-ki open set in 
$\Grass_r(H^1(X,\Q))$.
\end{corollary}

The straightness hypothesis is crucial for 
Theorem \ref{thm:straight omega} to hold. 
The next example shows the necessity of 
condition \eqref{v3} from Definition \ref{def:straight}.  

\begin{example}
\label{ex:straight c again}
Let $G$ be the group from Example \ref{ex:straight c}. 
Recall we have 
$\WW^1(G)=\{1\}\cup \{t \in (\C^{\times})^2 \mid t_1=-1\}$,  
but $\RR^1(G)=\{0\}$.  Thus, 
$\Omega^1_2(G)=\emptyset$, yet 
$\sigma_2(\RR^1(G,\Q))^{\compl}=\pt$. 
\end{example}

\section{The influence of formality}
\label{sec:formality}

An important property that bridges the gap between the tangent 
cone to a characteristic variety and the corresponding 
resonance variety is that of formality.  

\subsection{Formality}
\label{subsec:formal}

As before, let $X$ be connected CW-complex with finite 
$1$-skeleton.  In \cite{Su77}, Sullivan constructs 
an algebra $A_{\PL}(X, \Q)$ of polynomial differential forms 
on $X$ with coefficients in $\Q$, and provides it with a 
natural  commutative differential graded algebra (cdga) structure.  

Let $H^*(X,\Q)$ be the rational cohomology algebra 
of $X$, endowed with the zero differential.  
The space $X$ is said to be {\em formal}\/ if there is a 
zig-zag of cdga morphisms connecting $A_{\PL}(X, \Q)$ 
to $H^*(X,\Q)$, with each such morphism inducing an 
isomorphism in cohomology. The space $X$ is merely 
{\em $k$-formal}\/  (for some $k\ge 1$) if each of these 
morphisms induces an isomorphism in degrees up to $k$, 
and a monomorphism in degree $k+1$.

Examples of formal spaces include rational cohomology 
tori, surfaces, compact connected Lie groups, as well 
as their classifying spaces.   On the other hand, the 
only nilmanifolds which are formal are tori. Formality 
is preserved under wedges and products of spaces, 
and connected sums of manifolds.  

The $1$-minimality property of a space $X$ 
depends only on its fundamental group, $G=\pi_1(X,x_0)$.  
Alternatively, a finitely generated group $G$ is $1$-formal 
if and only if its Malcev Lie algebra admits a quadratic 
presentation.   Examples of $1$-formal groups include free 
groups and free abelian groups of finite rank, surface groups, 
and groups with first Betti number equal to $0$ or $1$.  
The $1$-formality property is preserved under free products 
and direct products.  

\subsection{The tangent cone formula}
\label{subsec:formal tc}

The main connection between the formality property 
and the cohomology jump loci is provided by the 
following theorem from \cite{DPS-duke}.  
For more details and references, 
we refer to the recent survey \cite{PS-bmssmr}.   
 
\begin{theorem}[\cite{DPS-duke}]
\label{thm:tcone}
Let $X$ be a $1$-formal space.  For each $d>0$,
the exponential map $\exp\colon H^1(X, \C) \to H^1(X, \C^{\times})$ 
restricts to an isomorphism of analytic germs, 
$\exp\colon (\RR^1_d(X),0)  \isom(\VV^1_d(X),1)$. 
Thus, the following ``tangent cone formula" holds:
\[
\tau_1(\VV^1_d(X))=\TC_1(\VV^1_d(X))=\RR^1_d(X).
\]
\end{theorem}

As a consequence, the irreducible components of 
$\RR^1_d(X)$ are all rationally defined subspaces, while 
the components of $\VV^1_d(X)$ passing through 
the origin are all rational subtori of the form $\exp(L)$, with $L$ 
running through the irreducible components of $\RR^1_d(X)$. 
The next corollary is immediate. 

\begin{corollary}
\label{cor:formal-ls}
Every $1$-formal space is locally $1$-straight.
\end{corollary}

In general, though, $1$-formal spaces need not be $1$-straight, 
as we shall see in Example \ref{ex:deleted B3}.  
Conversely, $1$-straight spaces need not be $1$-formal,  
as we shall see in Example \ref{ex:bundle}. 

\subsection{Formality and the Dwyer--Fried invariants}
\label{subsec:formal df}

The formality of a space has definite implications 
on the nature of its $\Omega$-invariants, and their 
relationship to the resonance varieties.

\begin{corollary}
\label{cor:df 1f}
If $X$ is $1$-formal, then 
$\Omega^1_r(X) \subseteq \sigma_r(\RR^1(X,\Q))^{\compl}$, 
for all $r\ge 1$. 
\end{corollary}

\begin{proof}
Follows from Corollaries \ref{cor:res bound} and \ref{cor:formal-ls}.
\end{proof}

\begin{corollary}
\label{cor:df tori}
Let $X$ be a $1$-formal space.  Suppose all positive-dimensional 
components of $\WW^1(X)$ pass through $1$.  Then:
\begin{enumerate}
\item \label{1f1}  $X$ is $1$-straight.
\item \label{1f2}   $\Omega^1_r(X) = \sigma_r(\RR^1(X,\Q))^{\compl}$, 
for all $r\ge 1$.
\end{enumerate}
\end{corollary}

\begin{proof}
Part \eqref{1f1} follows from Corollary \ref{cor:formal-ls} and 
the additional hypothesis. 
Part \eqref{1f2} now follows from Theorem \ref{thm:straight omega} 
(with $k=1$). 
\end{proof}

The hypothesis on the components of $\WW^1(X)$ is 
really needed in this corollary.  Indeed, if $X$ is the 
presentation $2$-complex for the group $G$ from 
Examples \ref{ex:straight c} and \ref{ex:straight c again}, 
then $H^*(X,\Q)\cong H^*(T^2,\Q)$, and so $X$ is formal.  
Yet, as we know,  $X$ is not $1$-straight, and 
$\Omega^1_2(X) \ne \sigma_2(\RR^1(X,\Q))^{\compl}$.  
We shall see another instance of this phenomenon in 
Example \ref{ex:deleted B3}. 

\subsection{The case of infinite cyclic covers}
\label{subsec:df1 formal}

In the case when $r=1$, Theorem \ref{thm:df cv} 
has the following consequence. 

\begin{corollary}[\cite{PS-plms}]
\label{cor:exp 1}
Let $\nu\colon \pi_1(X)\surj \Z$ be an epimorphism, and let 
$\bar\nu\in H^1(X,\Z)\subset H^1(X,\C)$ be the corresponding 
cohomology class.   If the exponential map restricts to 
an isomorphism of analytic germs, $(\RR^i_1(X),0)  
\cong (\VV^i_1(X),1)$, for all $i\le k$, then
\[
\sum_{i\le k} b_{i} (X^{\nu}) <\infty \same 
\bar\nu\not\in  \RR^{k}(X).
\]
\end{corollary}

Using Theorem \ref{thm:tcone} together with Corollary \ref{cor:exp 1}, 
we obtain the following immediate consequence. 

\begin{corollary}
\label{cor:df 1formal}
Suppose $G$ is a $1$-formal group. Then
\[
\Omega^1_1(G)=  \overline{\RR}^1(G,\Q)^{\compl}.
\]
\end{corollary}

In other words, if $X^{\nu}\to X$ is a regular, infinite cyclic 
cover of a $1$-formal space, then $b_1(X^{\nu})$ is finite 
if and only if the corresponding cohomology class, 
$\bar\nu \in H^1(X,\Z)$, is non-resonant.  

\section{Toric complexes}
\label{sec:toric}

In this section, we illustrate our techniques on a class 
of CW-com\-plexes that are carved out of an $n$-torus 
in a manner prescribed by a simplicial complex on 
$n$ vertices. Such ``toric" complexes are minimal, 
formal, and straight---thus, ideal from our point of view.

\subsection{Toric complexes and right-angled Artin groups}
\label{subsec:toric raag}

Let $T^n$ be the $n$-torus, endowed with the standard 
cell decomposition, and with basepoint $*$ the unique 
$0$-cell. For a simplex $\sigma \in [n]$, let $T^{\sigma}\subset T^n$ 
be the cellular subcomplex $T^{\sigma}= \{ x \in T^n \mid x_i = 
* \text{ if } i\notin \sigma\}$.  

\begin{definition}
\label{def:toric}
Let $K$ be a simplicial complex on $n$ vertices.  
The associated {\em toric complex}, $T_K$, is the union 
of all $T^{\sigma}$, with $\sigma$ running through the 
simplices of $K$.
\end{definition}

The $k$-cells of $T_K$ are in one-to-one 
correspondence with the $(k-1)$-simplices of $K$. 
Since the toric complex is a subcomplex of $T^{n}$, 
all boundary maps in $C_{\bullet}(T_K,\Z)$ vanish; 
thus, $T_K$ is a minimal cell complex. Evidently, 
$H_k(T_K,\Z)$ is isomorphic to $C^{\rm simp}_{k-1}(K,\Z)$,  
the free abelian group on the $(k-1)$-simplices of $K$.

Denote by $\sV$ be the set of $0$-cells of $K$, and 
by $\sE$ the set of $1$-cells of $K$.  The fundamental group 
of the toric complex $T_K$ is  then  the {\em right-angled 
Artin group}\/ associated to the graph $\G=K^{(1)}$, 
\begin{equation}
\label{eq:raag}
G_{\G} = \langle v \in \sV \mid vw = wv \text{ if } 
\{v,w\} \in \sE \rangle. 
\end{equation}
Groups of this sort interpolate between $G_{\G}=\Z^n$ in case 
$\G$ is a complete graph, and $G_{\G}=F_n$ in 
case $\G$ is a discrete graph. 

The toric complex construction behaves well with respect
to simplicial joins:  $T_{K*K'}=T_{K}\times T_{K'}$.  
Consequently, $G_{\G*\G'}=G_{\G}\times G_{\G'}$. 

A classifying space for the group $G_K$ is the toric complex 
$T_{\Delta}$, where $\Delta=\Delta_{K}$ is the flag complex 
of $K$, i.e., the maximal simplicial complex with 
$1$-skeleton equal to the graph $\G=K^{(1)}$. 

Finally, it is known from the work of Notbohm and Ray \cite{NR05}
that all toric complexes are formal spaces. In particular, 
all right-angled Artin groups are $1$-formal, a fact also 
proved in \cite{PS-mathann}.

\subsection{Cohomology jump loci}
\label{subsec:toric cjl}

As noted above, $H_1(T_K,\Z)=\Z^n$, with generators 
indexed by the vertex set $\sV=[n]$.  Thus, we may 
identify $H^1(T_K,\C)$ with the vector space $\C^{\sV}=\C^n$, 
and $H^1(T_K,\C^{\times})$ with the algebraic 
torus $(\C^{\times})^{\sV}=(\C^{\times})^n$.  For each 
subset $\sW \subseteq \sV$, let $\C^{\sW}$ be 
the respective coordinate subspace, and let 
$(\C^{\times})^{\sW}=\exp(\C^{\sW})$ be the 
respective algebraic subtorus.  

\begin{theorem}[\cite{PS-adv}]
\label{thm:cjl tc}
With notation as above,
\begin{equation*}
\label{eq:cjl tc}
\RR^i_d(T_K)= \bigcup_{\sW} \, 
\C^{\sW} 
\quad \text{and}  \quad
\VV^i_d(T_K)= \bigcup_{\sW} \, 
(\C^{\times})^{\sW},
\end{equation*}
where, in both cases, the union is taken over 
all subsets $\sW\subset\sV$ for which 
$\sum_{\sigma\in K_{\sV\setminus \sW}}
\dim_{\C} \widetilde{H}_{i-1-\abs{\sigma}} 
(\lk_{K_\sW}(\sigma),\C) \ge d$.
\end{theorem}

In the above, $K_\sW$ denotes the simplicial subcomplex 
induced by $K$ on $\sW$, and $\lk_{L}(\sigma)$ denotes 
the link of a simplex $\sigma$ in a subcomplex $L\subseteq K$.

In homological degree $1$, the resonance formula from 
Theorem \ref{thm:cjl tc} takes a simpler form, 
already noted in \cite{PS-mathann}.  Namely, 
$\RR^1(G_{\G}) = \bigcup_{\sW}  \C^{\sW}$, 
where the union is taken over all (maximal) subsets 
$\sW\subset\sV$ for which the induced graph
$\G_{\sW}$ is disconnected.  In particular, 
the codimension of the resonance variety 
$\RR^1(G_{\G})$ equals the connectivity 
of the graph $\G$. 

\begin{corollary}
\label{cor:straight toric}
All toric complexes $T_K$ are straight spaces.
\end{corollary}

\begin{proof}
By the above theorem, each resonance variety 
$\RR^i(T_K)$ is the union of a coordinate subspace arrangement, 
and each characteristic variety $\VV^i(T_K)$ is the union of the 
corresponding arrangement of coordinate subtori.  Consequently, 
all components of $\VV^i(T_K)$ are algebraic subtori, 
and $\TC_1(\VV^i(T_K))= \RR^i(T_K)$. 
\end{proof}

\subsection{$\Omega$-invariants}
\label{subsec:toric omega}

In their landmark paper \cite{BB}, Bestvina and Brady studied 
the geometric finiteness properties of certain subgroups of 
right-angled Artin groups $G_{\G}$, arising as kernels of 
``diagonal" homomorphisms $G_{\G}\surj \Z$.  In \cite{PS-adv}, 
Papadima and Suciu computed the homology of such 
subgroups, and, more generally, the homology of 
regular $\Z$-covers of toric complexes.  In a related vein, 
Denham \cite{De} investigated the homology of covers 
of toric complexes $T_K$ corresponding to 
``coordinate" homomorphisms $\pi_1(T_K)\surj \Z^r$. 

The study of homological finiteness properties of regular, 
free abelian covers of toric complexes was pursued in \cite{PS-plms},   
where a general formula for the Dwyer--Fried sets of 
such complexes was given.  In our setting, this result 
may be restated as follows.

\begin{theorem}[\cite{PS-plms}]
\label{thm:df toric}
Let $T_K$ be a toric complex. Then 
\[
\Omega^k_r(T_K) = \sigma_r(\RR^k(T_K,\Q))^{\compl}, 
\]
for all $k, r\ge 1$. In particular, $\Omega^k_1(T_K) =  
\overline\RR^k(T_K,\Q)^{\compl}$.
\end{theorem}

\begin{proof}
Follows from Theorem \ref{thm:straight omega} and 
Corollary \ref{cor:straight toric}. 
\end{proof}

As a consequence, all the $\Omega$-invariants of a toric complex $T_K$ 
are Zariski open subsets of the Grassmannian $\Grass_r(\Q^n)$.

When combined with Theorem \ref{thm:cjl tc} and the discussion 
following it,  Theorem \ref{thm:df toric} allows 
us to compute very explicitly the Dwyer--Fried sets of toric 
complexes.  Let us provide one such computation. 

\begin{corollary}
\label{cor:df graph}
Let $\G$ be a finite simple graph, and let $\kappa$ be 
the connectivity of $\G$.  Then $\Omega^1_r(G_{\G})=\emptyset$,
for all $r\ge \kappa+1$. 
\end{corollary}

\begin{proof}
Recall that $\codim \RR^1(G_{\G})=\kappa$. 
Thus, $\codim \sigma_r(\RR^1(G_{\G})) =\kappa-r+1$. 
The desired conclusion follows from Theorem \ref{thm:df toric}.
\end{proof}

In particular, if $\G$ is disconnected, then 
$\Omega^1_r(G_{\G})=\emptyset$, for all $r\ge 1$.

\begin{example}
\label{ex:tree}
Let $\G$ be a tree on $n\ge 3$ vertices. Label the 
non-terminal vertices as $v_1,\dots , v_s$, and the 
terminal vertices as $v_{s+1},\dots ,v_n$.  The cut 
sets of $\G$ are all singletons, consisting of the  
non-terminal vertices.  Thus, the resonance variety 
$\RR^1(G_\G,\Q)$ is the union of the coordinate 
hyperplanes $L_j=\{z\in \Q^n\mid z_j=0\}$, with $1\le j\le s$.  Hence,
\begin{equation}
\label{eq:omega tree}
\Omega^1_r(G_{\G})=
\begin{cases}
\QP^{n-1} \setminus \bigcup_{j=1}^{s} \overline{L}_j 
& \text{if $r=1$},\\
\emptyset & \text{if $r\ge 2$}.
\end{cases}
\end{equation}
\end{example}

\section{K\"{a}hler and quasi-K\"{a}hler manifolds}
\label{sec:kahler}

We now discuss the cohomology 
jumping loci and the Dwyer--Fried invariants of 
K\"{a}hler and quasi-K\"{a}hler manifolds. 

\subsection{Cohomology ring and formality}
\label{subsec:coho kahler}

Let $M$ be a compact, connected, complex manifold 
of complex dimension $m$.  Such a manifold is called 
a {\em K\"{a}hler manifold}\/ if it admits a Hermitian 
metric $h$ for which the imaginary part $\omega=\Im(h)$ 
is a closed $2$-form.   The class of K\"{a}hler manifolds, 
includes smooth, complex projective varieties, such 
as Riemann surfaces.  This  class is closed under 
finite direct products and finite covers. 

Hodge theory provides two sets of data on the cohomology 
ring of $M$. The first data, known as the Hodge decomposition 
on $H^i (M,\C)$, depend only on the complex structure on $M$.  
The second data, known as the Lefschetz isomorphism 
and the Lefschetz decomposition on $H^i (M,\R)$, 
depend only on the choice of a K\"{a}hler class 
$[\omega]\in  H^{1,1}(M,\C)$. 

These data impose strong conditions on the possible 
Betti numbers $b_i=b_i(M)$, beyond the 
symmetry property $b_i=b_{2m-i}$ imposed by 
Poincar\'{e} duality.   
For example, the odd Betti numbers $b_{2i+1}$ 
must be even, and increasing in the range 
$2i+1\le m$, while the even Betti numbers 
$b_{2i}$ must be increasing in the range $2i\le m$.  

Another constraint on the topology of compact 
K\"{a}hler manifolds was established by Deligne, Griffiths, 
Morgan, and Sullivan in \cite{DGMS}.   For such a 
manifold $M$, let $d$ be the exterior derivative, 
$J$ the complex structure, and $d^c=J^{-1}dJ$. Then the 
following holds: If $\eta$ is a form which is closed for 
both $d$ and $d^c$, and exact for either $d$ or $d^c$, 
then $\eta$ is exact  for $dd^c$.  As a consequence, 
all compact K\"{a}hler manifolds are formal. 

A manifold $X$ is said to be a {\em quasi-K\"{a}hler manifold}\/ 
if there is a compact K\"{a}hler manifold $\overline{X}$ and a 
normal-crossings divisor $D$ such that $X=\overline{X}\setminus D$. 
The class of quasi-K\"{a}hler manifolds includes 
smooth, irreducible, quasi-projective complex varieties, 
such as complements of plane algebraic curves. 

Each quasi-K\"{a}hler manifold $X$ inherits a mixed 
Hodge structure from its compactification $\overline{X}$.  
If $X$ is a smooth, quasi-projective variety 
with vanishing degree $1$ weight filtration on $H^1(X,\C)$, 
then $X$ is $1$-formal.  This happens, for instance, 
when $X$ admits a non-singular compactification 
$\overline{X}$ with $b_1(\overline{X})=0$, e.g.,  
when $X$ is the complement of a hypersurface 
in $\CP^{m}$.  In general, though, smooth, 
quasi-projective varieties need not be $1$-formal.  
For a detailed treatment of the subject, we refer 
to Morgan \cite{Mo}. 

\subsection{Characteristic varieties}
\label{subsec:cv kahler}

Foundational results on the structure of the cohomology 
support loci for local systems on smooth projective varieties, 
and more generally, on compact K\"{a}hler manifolds were 
obtained by  Beauville, Green--Lazarsfeld, Simpson, and 
Campana.   A further, wide-ranging generalization was 
obtained by Arapura in \cite{Ar}.  
 
\begin{theorem}[\cite{Ar}]
\label{thm:arapura}
Let $X=\overline{X}\setminus D$, where $\overline{X}$ 
is a compact K\"{a}hler manifold and $D$ is a normal-crossings 
divisor. If either $D=\emptyset$ or $b_1(\overline{X})=0$, 
then each characteristic variety $\VV^i_d(X)$ 
is a finite union of unitary translates of algebraic 
subtori of $H^1(X,\C^{\times})$. 
\end{theorem}

In other words, for each $i\ge 0$ and $d>0$, the 
characteristic variety $\VV^i_d(X)$ admits a 
decomposition into irreducible components of the form
$W= \rho \cdot T$, with
\begin{itemize}
\item 
direction subtorus $T=\dire( W)$, a connected algebraic 
subgroup of the character torus $\wG^0$, where $G=\pi_1(X)$;
\item  
translation factor $\rho \colon G \to S^1$, a unitary 
character of $G$. 
\end{itemize}

In degree $1$ and depth $1$, the condition that 
$b_1(\overline{X})=0$ if $D\ne \emptyset$ may be lifted.  
Furthermore, each positive-dimensional component 
of $\VV^1_1(X)$ is of the form $\rho \cdot T$, 
with $\rho$ a {\em torsion}\/ character. 

In the  quasi-projective setting, more can be said. 
The next theorem summarizes several recent results 
in this direction: the first two parts are from \cite{DPS-imrn}, 
the third part is from \cite{Bu} and \cite{ACM}, 
while the last part is from \cite{Di07} and \cite{ACM}.

\begin{theorem}[\cite{Di07, DPS-imrn, Bu, ACM}]
\label{thm:qp cv}
Let $X$ be a smooth, quasi-projective variety.  Then:

\begin{enumerate}
\item \label{a1}
If $W$ and $W'$ are two distinct components of 
$\VV^1(X)$, then either $\dire (W) = \dire (W')$, or 
$T_1 \dire (W) \cap T_1 \dire (W')= \{ 0\}$. 

\item \label{a2}
For each pair of distinct components, $W$ and $W'$, the 
intersection $W\cap W'$ is a finite set of torsion characters. 

\item \label{a3}  
The isolated points in $\VV^1(X)$ are also torsion 
characters.

\item \label{a4}  
If $W=\rho T$, with $\dim T=1$ and $\rho\ne 1$, 
then $T$ is not a component of $\VV^1(X)$.
\end{enumerate}
\end{theorem}

\subsection{Resonance varieties}
\label{subsec:res kahler}

In the presence of $1$-formality, the quasi-K\"{a}hler 
condition also imposes stringent conditions 
on the degree~$1$ resonance varieties. 

\begin{theorem}[\cite{DPS-duke}]
\label{thm:res quasi-kahler} 
Let $X$ be a $1$-formal, quasi-K\"{a}hler manifold, and let 
$\{ L_{\alpha}\}$ be the collection of positive-dimensional, 
irreducible components of $\RR^1_1(X)$.  Then:
\begin{enumerate}
\item \label{rqk1} Each $L_{\alpha}$ is a linear 
subspace of $H^1(X, \C)$ of dimension at least 
$2\varepsilon(\alpha)+2$, 
for some $\varepsilon(\alpha)\in \{0,1\}$.
\item \label{rqk2} The restriction of the cup-product 
map $H^1(X,\C)\wedge H^1(X,\C) \to H^2(X,\C)$ 
to $L_{\alpha} \wedge L_{\alpha}$ has rank $\varepsilon(\alpha)$.
\item \label{rqk3} If $\alpha \ne \beta $, then 
$L_{\alpha} \cap L_{\beta}=\{0\}$.
\item  \label{rqk4} 
$\RR^1_d(X)=\{0\} \cup \bigcup_{\alpha :  
d\le \dim L_{\alpha} -\varepsilon(\alpha)} 
L_{\alpha}$.  
\end{enumerate}
\end{theorem}

\begin{remark}
\label{rem:qpres}
Suppose $X$ is a smooth quasi-projective variety 
with $W_1(H^1(X,\C))=0$.  
Then, as mentioned in \S\ref{subsec:coho kahler}, $X$ is  
$1$-formal.  In this situation, each subspace $L_{\alpha}$ 
is isotropic, i.e., $\epsilon(\alpha)=0$. 
\end{remark}

\begin{remark}
\label{rem:kres}
Suppose $M$ is a compact K\"{a}hler manifold.   
Then of course $M$ is formal, and Theorem \ref{thm:res quasi-kahler} 
again applies.  In this situation, each subspace $L_{\alpha}$ 
has dimension $2g(\alpha)$, for some $g(\alpha)\ge 2$, and the 
restriction of the cup-product map to $L_{\alpha} \wedge L_{\alpha}$ 
has rank $\epsilon(\alpha)=1$. 
\end{remark}

It is now a straightforward exercise to enumerate 
the possibilities for the first resonance variety of a 
K\"{a}hler manifold $M$, at least for low values of $n=b_1(M)$:
\begin{itemize}
\item If $n=0$ or $2$, then $\RR^1(M)=\{0\}$.
\item If $n=4$, then $\RR^1(M)=\{0\}$ or $\C^4$.
\item If $n=6$, then $\RR^1(M)=\{0\}$, $\C^4$,  or 
$\C^6$.
\item If $n=8$, then $\RR^1(M)=\{0\}$, $\C^4$, 
$\C^6$, $\C^8$, or $\RR^1(M)$ consists of two transverse, 
$4$-dimensional subspaces.
\end{itemize}

Using the computations from Examples \ref{ex:res torus} 
and \ref{ex:prod surf}, it is readily seen 
that all these possibilities can be realized by manifolds 
of the form $M=S_g$ or $M=S_g\times S_h$, with $S_g$ 
and $S_h$ Riemann surfaces of suitable genera $g,h\ge 0$. 

\subsection{Straightness} 
\label{subsec:straight kahler}
We now discuss in detail the straightness properties of 
K\"{a}hler and quasi-K\"{a}hler manifolds, and how these  
properties relate to $1$-formality.

\begin{prop}
\label{prop:straight kahler}
Let $X$ be a $1$-formal, quasi-K\"{a}hler manifold 
(for instance, a compact K\"{a}hler manifold).  Then:
\begin{enumerate}
\item \label{sk1}
$X$ is locally $1$-straight.
\item \label{sk2}
$X$ is $1$-straight if and only if $\WW^1(X)$ contains no 
positive-dimensional translated subtori.
\end{enumerate} 
\end{prop}

\begin{proof}
Part \eqref{sk1} follows from Corollary \ref{cor:formal-ls}, 
using only the $1$-formal\-ity assumption, while part \eqref{sk2}
follows from Theorem \ref{thm:arapura}.
\end{proof}

In general, quasi-K\"{a}hler manifolds may fail to be locally 
$1$-straight, as illustrated in Example \ref{ex:cx heisenberg} 
below.  Perhaps more surprisingly, compact K\"{a}hler manifolds 
may fail to be $1$-straight; for a concrete example, we refer 
to \cite{Su}.

\begin{example}
\label{ex:cx heisenberg}
Let $X$ be the complex Heisenberg manifold, 
i.e., the total space of the $\C^{\times}$-bundle over 
$\C^{\times}\times \C^{\times}$ with Euler number~$1$. 
Then $X$ is a smooth, quasi-projective variety with $\VV^1(X)=\{1\}$, 
yet $\RR^1(X)=\C^2$.  Thus, the tangent cone formula 
fails in this instance, and so $X$ is neither $1$-formal, nor 
locally $1$-straight.
\end{example}

Let $(Y,0)$ be a quasi-homogeneous isolated surface singularity. 
Then $X=Y \setminus \{0\}$ is a smooth, quasi-projective 
variety which supports a ``good" $\C^{\times}$-action, 
with orbit space a smooth projective curve.   Moreover, 
$X$ deform-retracts onto the singularity link, which is 
a closed, smooth, orientable $3$-manifold.  

\begin{prop}
\label{prop:tc non1f}
Suppose the curve $X/\C^{\times}$ has genus $g>1$.  
Then $X$ is straight, yet $X$ is not $1$-formal. 
\end{prop}

\begin{proof}
According to formula (14) from \cite{DPS-mz}, 
we have  $\VV^1(X)=H^1(X,\C^{\times})$. Hence, 
by Lemma \ref{lem:straight}, $X$ is straight.  
The fact that $X$ is not $1$-formal follows from 
\cite[Proposition 7.1]{DPS-mz}.
\end{proof}

We illustrate this result with a concrete family of examples. 

\begin{example}
\label{ex:bundle}
Let $X$ be the total space of the $\C^{\times}$-bundle 
over the Riemann surface of genus $g$, with Euler 
number $-1$. Then $X$ is homotopy equivalent to 
the Brieskorn manifold $\Sigma(2, 2g+1, 2(2g+1))$.  
Assume now $g>1$; then clearly $X$ fits into the 
setup of Proposition \ref{prop:tc non1f}.  Hence, 
$X$ is straight, but not $1$-formal. 
\end{example}

\subsection{Dwyer--Fried invariants} 
\label{subsec:df kahler}

We conclude this section with a discussion of the 
$\Omega$-invariants of (quasi-) K\"{a}hler manifolds, 
and the extent to which these invariants are determined 
by the resonance varieties.

\begin{theorem}
\label{thm:df kahler}
Let $X$ be a $1$-formal, quasi-K\"{a}hler manifold 
(for instance, a compact K\"{a}hler manifold). Then:
\begin{enumerate}
\item \label{qk1}
$\Omega^1_1(X) =\overline{\RR}^1(X,\Q)^{\compl}$ 
and 
$\Omega^1_r(X) \subseteq \sigma_r(\RR^1(X,\Q))^{\compl}$, 
for $r\ge 2$. 
\item \label{qk2}
If $\WW^1(X)$ contains no positive-dimensional 
translated subtori, then 
$\Omega^1_r(X) = \sigma_r(\RR^1(X,\Q))^{\compl}$, 
for all $r\ge 1$. 
\end{enumerate} 
\end{theorem}

\begin{proof}
Part \eqref{qk1}.  Using only the $1$-formality assumption, 
the two statements follow from Corollaries \ref{cor:df 1formal} 
and \ref{cor:df 1f}, respectively.

Part \eqref{qk2}.  In view of Theorem \ref{thm:arapura}, 
the hypothesis is equivalent to $\WW^1(X)$ containing 
no positive-dimensional component not passing through $1$. 
The conclusion follows from Corollary \ref{cor:df tori}.
\end{proof}

In general, though, the inclusion 
$\Omega^1_r(X) \subseteq \sigma_r(\RR^1(X,\Q))^{\compl}$
can be strict, provided $2\le r\le b_1(X)$. 
The next theorem identifies a fairly broad class of 
smooth, quasi-projective varieties for which this is the case. 
An explicit example will be given at the end of 
Section \ref{sec:arrs}.

\begin{theorem}
\label{thm:tt}
Let $X$ be a $1$-formal, smooth, quasi-projective variety. 
Suppose 
\begin{enumerate}
\item \label{h1} $\WW^1(X)$ has a $1$-dimensional component 
not passing through $1$;
\item  \label{h2} $\RR^1(X)$ has no codimension-$1$ components.
\end{enumerate}
Then $\Omega^1_{2}(X)$ is 
strictly contained in $\sigma_{2}(\RR^1(X,\Q))^{\compl}$. 
\end{theorem}

\begin{proof}
Set $n=b_1(X)$, and identify $H=H^1(X,\Q)$ with $\Q^n$.
By assumption  \eqref{h1}, the characteristic variety 
$\WW^1(X)\subset (\C^{\times})^n$ has a 
component of the form $W=\rho\cdot T$, with 
\begin{itemize}
\item $T=\exp(\ell \otimes \C)$, where $\ell$ is a $1$-dimensional  
subspace in $H$;
\item $\rho = \exp(2\pi i q)$, where $q$ is a vector in $H\setminus \ell$.
\end{itemize}

We claim that the line $\ell$ is not contained in 
the resonance variety $\RR^1(X,\Q)$. For, if it were, the 
$1$-dimensional algebraic torus $T$ would be contained 
in $\WW^1(X)$.  But we know from 
Theorem \ref{thm:qp cv}\eqref{a4} that 
$T$ is not a component of $\WW^1(X)$.  Hence, 
there would exist a component $W'$ 
with $T\subsetneqq W'$.  Therefore, 
\begin{align*}
&\dire(W)=T\ne W'=\dire (W'),\ \: \text{and}\\ 
&T_1(\dire(W))\cap T_1(\dire (W'))=\ell\otimes \C\ne \{0\}.
\end{align*}
This contradicts Theorem \ref{thm:qp cv}\eqref{a1}, thereby 
establishing the claim. 

In view of the above, and of assumption \eqref{h2}, the 
resonance variety $\RR^1(X,\R)$ has codimension 
at least $2$ in $H_{\R}=H^1(X,\R)$. Therefore, the set 
\[
Z=\{ P \in \Grass_2(H_{\R}) \mid \ell \subset P 
\text{ and } P \cap \RR^1(X,\R) \ne \{0\} \}
\] 
is a proper subvariety of $\Grass_2(H_{\R})$.  
Hence, there is a non-zero vector $r_0\in \R^{n}$, 
and an open cone $U$ containing $r_0$, such that, 
for all $r\in U$, the plane $P$ spanned by $r$ and $\ell$ 
intersects $\RR^1(X,\R)$ only at $0$.

Let $\pi\colon \R^n\setminus \{0\} \to \RP^{n-1}$ be the 
projection map. Clearly, $\pi(\Z^n\setminus \{0\})$ is a 
dense subset of $\RP^{n-1}$.  Thus, $\pi(q+\Z^n)$ is 
also dense, and so intersects $\pi(U)$.  
Hence, there is a lattice point $\lambda\in \Z^n$ 
such that $\pi(q+\lambda)$ belongs to $\pi(U)$. 
The rational vector  $q_0:=q+\lambda$ then belongs to $U$.

Let $P_0$ be the $2$-dimensional subspace of $H$ spanned 
by $\ell$ and $q_0$.   By construction, $P_0\cap \RR^1(X,\Q)=\{0\}$, 
and so $P_0\in \sigma_{2}(\RR^1(X,\Q))^{\compl}$.  
On the other hand, the algebraic torus 
$T_0=\exp(P_0\otimes \C)$ contains both 
$\exp(\ell \otimes \C)=T$ and $\exp(2\pi i q_0)=\rho$; 
therefore, $T_0\supset \rho T$. Consequently, 
$\dim (T_0\cap \WW^1(X))>0$, showing that 
$P_0\notin \Omega^1_{2}(X)$. 
\end{proof}

\section{Hyperplane arrangements}
\label{sec:arrs}

We conclude with a class of spaces exhibiting a strong  
interplay between the resonance varieties and the finiteness 
properties of free abelian covers.  
These spaces, obtained by deleting a finite number 
of hyperplanes from a complex affine space, are  
minimal, formal, and locally straight, but not always straight. 

\subsection{Cohomology jump loci} 
\label{subsec:cv arr}

Let $\A$ be an arrangement of hyperplanes in $\C^{\ell}$. 
To start with, we will assume that all hyperplanes 
in $\A$ pass through the origin; non-central arrangements 
can be handled much the same way, using standard 
coning and deconing constructions. 

Let $X=X(\A)$ be the complement of the union 
of the hyperplanes comprising $\A$.  Then $X$ can be 
viewed as the complement of a normal-crossing divisor 
in a suitably blown-up $\CP^{\ell}$.   In particular, 
$X$ has the homotopy type of an $\ell$-dimensional 
CW-complex.  Moreover, this CW-complex can be chosen 
to be minimal (Dimca--Papadima, Randell). 

Let $G=G(\A)$ be the fundamental group of the 
complement.  Its abelianization, $G_{\ab}$, is the free 
abelian group of rank $n=\abs{\A}$.  Thus, we may identify 
the character group $\wG$ with the complex algebraic 
torus  $(\C^{\times})^n$.   Let $\VV^i_d(\A)=\VV^i_d(X(\A))$ 
be the characteristic varieties of the arrangement.  
By Arapura's work \cite{Ar}, these varieties consist 
of subtori of  $(\C^{\times})^n$, possibly translated 
by unitary characters, together with a finite number 
of isolated unitary characters.   

The cohomology ring $A=H^*(X(\A),\Z)$ was computed 
by Brieskorn in the early 1970s, building on pioneering 
work of Arnol'd on the cohomology ring of the braid arrangement.  
It follows from Brieskorn's work that the space $X$ is formal; 
in particular, the fundamental group of $X$ is $1$-formal. 
In 1980, Orlik and Solomon gave a simple description 
of the ring $A$, solely in terms of the intersection 
lattice $L(\A)$, i.e., the poset of all intersections 
of $\A$, ordered by reverse inclusion.  

The resonance varieties $\RR^i_d(\A)=\RR^i_d(X(\A))$ 
were first defined and studied by Falk in \cite{Fa97}.   
Identifying $H^1(X,\C)$ with $\C^n$, we may view the 
resonance varieties of $\A$ as homogeneous 
subvarieties of $\C^n$. It is known from work 
of Yuzvinsky that these varieties actually lie in the 
hyperplane $\{x\in \C^n \mid \sum_{i=1}^{n} x_i=0\}$. 
Moreover, it follows from \cite{EPY} that resonance 
``propagates" for the Orlik-Solomon algebra.  More 
precisely, if $i\le k \le \ell$, then 
$\RR^{i}_1(\A) \subseteq \RR^{k}_1(\A)$;  
in particular, $\RR^{k}(\A)=\RR^k_1(\A)$.  

\subsection{Straightness} 
\label{subsec:arr straight}
As above, let $X=X(\A)$ be an arrangement complement.   
Using work of Esnault, Schechtman, and 
Viehweg \cite{ESV}, one can show that the exponential 
map, $\exp\colon H^1(X,\C) \to H^1(X,\C^{\times})$, 
induces an isomorphism of analytic germs from  
$(\RR^i_d(X),0)$ to $(\VV^i_d(X),1)$, 
for all $i\ge 0$ and $d>0$. We then have 
\begin{equation}
\label{eq:arr tcone}
\TC_1(\VV^i_d(\A))=\RR^i_d(\A), \quad\text{for all $d>0$}. 
\end{equation}
In particular, all the resonance 
varieties $\RR^i_d(\A)$ are finite unions of rationally 
defined linear subspaces. 

The tangent cone formula \eqref{eq:arr tcone} was 
first proved in degree $i=1$ (using different methods) 
by Cohen and Suciu \cite{CS99} and Libgober \cite{Li02}, 
and was generalized to the higher-degree jump loci by 
Cohen and Orlik \cite{CO00}. 

\begin{prop}
\label{prop:arr straight}
Let $\A$ be a hyperplane arrangement in 
$\C^{\ell}$, with complement $X$. 
\begin{enumerate}
\item \label{p1} 
$X$ is locally straight.
\item \label{p2} 
$X$ is $k$-straight if and only if $\VV^k(X)$ contains 
no positive-di\-mensional translated tori.  
\end{enumerate}
\end{prop}

\begin{proof}
For part \eqref{p1}, we must verify the two properties 
from Definition \ref{def:loc straight}. Property \eqref{v1} 
follows at once from Arapura's theorem \ref{thm:arapura}, 
while property \eqref{v2} follows from the tangent cone 
formula \eqref{eq:arr tcone}.

For part \eqref{p2}, we must verify the additional property 
\eqref{v3} from Definition \ref{def:straight}.   The conclusion 
follows again from Theorem \ref{thm:arapura}.
\end{proof}

As first noted in \cite{S01}, there do exist 
arrangements $\A$ for which $\VV^1(\A)$ contains 
positive-dimensional translated components.  
By Proposition \ref{prop:arr straight}, such arrangements 
are not $1$-straight. We will come back to this point 
in Example \ref{ex:deleted B3}. 

\subsection{Dwyer--Fried invariants} 
\label{subsec:df arr}

Let us define the Dwyer--Fried invariants of an arrangement 
$\A$ as $\Omega^{i}_{r}(\A)=\Omega^{i}_{r}(X(\A))$. 
Suppose $\A$ consists of $n$ hyperplanes in $\C^{\ell}$.  
For each $1 \le r\le n$, there is then a filtration 
\begin{equation}
\label{eq:arr df filt}
\Grass_r(\Q^n) = \Omega^0_r(\A) \supseteq \Omega^1_r(\A)  
\supseteq \Omega^2_r(\A)  \supseteq \cdots \supseteq \Omega^{\ell}_r(\A).
\end{equation}

The next result establishes a comparison between 
the terms of this filtration and the incidence varieties to 
the resonance varieties of $\A$.

\begin{theorem}
\label{thm:df arr}
Let $\A$ be an arrangement of $n$ hyperplanes, 
and fix an integer $k\ge 1$. Then, 
\begin{enumerate}
\item \label{ar1} 
$\Omega^k_r(\A) \subseteq 
\Grass_r(\Q^n) \setminus \sigma_r(\RR^k(\A,\Q))$, 
for all $r\ge 1$. 

\item \label{ar2} 
$\Omega^1_1(\A)=\QP^{n-1}\setminus \overline{\RR}^1(\A,\Q)$.  
 
\item \label{ar3} 
If $\VV^k(\A)$ contains no positive-dimensional 
translated tori, then 
$\Omega^k_r(\A) =\Grass_r(\Q^n) \setminus \sigma_r(\RR^k(\A,\Q))$,
for all $r\ge 1$. 
\end{enumerate}
\end{theorem}

\begin{proof}
Part \eqref{ar1} follows from Proposition \ref{prop:arr straight}\eqref{p1}  
and Corollary \ref{cor:res bound}.  

Part \eqref{ar2} follows from the $1$-formality of the 
complement of $\A$ and Corollary \ref{cor:df 1formal}.

Part \eqref{ar3} follows from 
Proposition \ref{prop:arr straight}\eqref{p2}  and 
Theorem \ref{thm:straight omega}.
\end{proof}

In other words, each Dwyer--Fried invariant $\Omega^k_r(\A)$ 
is included in the complement of a union of special Schubert 
varieties of the form $\sigma_r(L)$, where $L$  runs through 
the components of $\RR^k(\A,\Q)$, with the inclusion being an 
equality when $\VV^k(\A)=\bigcup_{L} \exp(L\otimes \C)$. 

\begin{remark}
\label{rem:falk}
In \cite{Fa04}, Falk gives a decomposition of 
the resonance variety $\RR^1(\A)$ into combinatorial 
pieces and shows that, projectively, each of these
pieces is the ruled variety corresponding to an 
intersection of special Schubert varieties in special 
position in the Grassmannian of lines in projective space. 
It would be interesting to see if Falk's description 
sheds additional light on the Dwyer--Fried invariants 
$\Omega^1_r(\A)$.
\end{remark}

\subsection{Resonance varieties of line arrangements}
\label{subsec:line arr}

For the rest of this section, we will concentrate on the 
resonance varieties $\RR^1(\A)=\RR^1(G(\A))$ and 
their relation to the Dwyer--Fried invariants 
$\Omega^1_r(\A)=\Omega^1_r(G(\A))$.   We 
start with a brief review of the former. 

By the Lefschetz-type theorem of Hamm and L\^{e}, 
taking a generic two-dimensional section does not 
change the group of the arrangement. 
Thus, we may assume $\A=\{\ell_1,\dots ,\ell_n\}$ is an 
affine line arrangement in $\C^2$, for which no two lines are 
parallel.  

The variety $\RR^1(\A)$  is a union of linear subspaces 
in $\C^n$. Each subspace has dimension at least $2$, 
and each pair of subspaces meets transversely at $0$.
The simplest components of $\RR^1(\A)$ are the 
{\em local}\/ components:  to an intersection point 
$v_J=\bigcap_{j\in J} \ell_j$ of multiplicity $\abs{J} \ge 3$, 
there corresponds a subspace $L_J$ of dimension $\abs{J}-1$, 
given by equations of the form $\sum_{j\in J} x_j = 0$, and 
$x_i=0$ if $i\notin J$.  The remaining components correspond 
to certain ``neighborly partitions" of sub-arrange\-ments of $\A$.  

If $\abs{\A}\le 5$, then all components of $\RR^1(\A)$ are local.  
For $\abs{\A}\ge 6$, though, the resonance variety $\RR^1(\A)$ 
may have non-local components. 

\begin{example}
\label{ex:braid}
Let $\A$ be the braid arrangement, with defining polynomial 
$Q(\A)=z_0z_1z_2(z_0-z_1)(z_0-z_2)(z_1-z_2)$. 
Take a generic plane section, and label the 
corresponding lines as $6,2,4,3,5,1$. 
The variety $\RR_{1}(\A)\subset \C^6$ has $4$ 
local components, corresponding to the triple 
points $124, 135, 236, 456$, 
and one non-local  component, 
$L_{\Pi}=\{ x_1+x_2+x_3=x_1- x_6=x_2-x_5=x_3-x_4=0 \}$, 
corresponding to the neighborly partition $\Pi=(16| 25 | 34)$.
\end{example}

For an arbitrary arrangement $\A$, work of 
Falk, Pereira, and Yuzvinsky \cite{FY, PeY, Yu} 
shows that any non-local component in $\RR^{1}(\A)$ 
has dimension either $2$ or $3$. 

\subsection{$\Omega$-invariants of line arrangements}
\label{subsec:df line arr}

We are in a position now where we can compute explicitly the   
Dwyer--Fried invariants of some line arrangements.  We start 
with a simple example.

\begin{example}
\label{ex:near pencil}  
Let $\A$ be the arrangement with defining polynomial 
$Q(\A)=z_0z_1z_2(z_1-z_2)$.  The variety 
$\RR^{1}(\A)\subset \C^4$ has a single component, 
namely, the $2$-plane $L=\{x\mid x_1+x_2+x_3=x_4=0\}$, 
while $\VV^1(\A)=\exp(L)$. Using Theorem \ref{thm:df arr}, 
we obtain
\[
\Omega^1_r(\A)=
\begin{cases}
\QP^3 \setminus \overline{L} & \text{if $r=1$},\\
\Grass_2(\Q^4) \setminus \sigma_2(L) & \text{if $r=2$},\\
\emptyset & \text{if $r\ge 3$}.
\end{cases}
\]
Here, the Grassmannian $\Grass_2(\Q^4)$ is the 
hypersurface in $\bP(\bigwedge^2 \Q^4)$ with equation 
$p_{12}p_{34}-p_{13}p_{24}+p_{23}p_{14}=0$, 
while the Schubert variety $\sigma_2(L)$ is the $3$-fold in 
$\Grass_2(\Q^4)$ cut out by the hyperplane 
$p_{12}-p_{13}+p_{23}=0$. 
\end{example}

\begin{prop}
\label{prop:df nr}
Let $\A$ be an arrangement of $n$ lines in $\C^{2}$, 
and let $m$ be the maximum multiplicity of its 
intersection points.
\begin{enumerate}
\item \label{z1}
If $m=2$, then $\Omega^1_r(\A) =  \Grass_r(\Q^n)$, 
for all $r\ge 1$.
\item \label{z2}
If $m\ge 3$, then  
$\Omega^1_r(\A) =  \emptyset$, for all $r \ge n - m + 2$.
\end{enumerate}
\end{prop}

\begin{proof}
If $\A$ has only double points, then $G(\A)=\Z^n$,   
by a well-known theorem of Zariski. Assertion \eqref{z1} 
follows from Example \ref{ex:torus}.  

Now suppose $\A$ has an intersection point of multiplicity $m\ge 3$.  
Then $\RR^1(\A)$ has a (local) component $L$ of dimension 
$m-1$.  The corresponding Schubert variety, $\sigma_r(L)$, 
has codimension $n-m+2-r$.  Assertion \eqref{z2}  follows 
from Theorem \ref{thm:df arr}.
\end{proof}

\begin{prop}
\label{prop:df 12}
Let $\A$ be an arrangement of $n$ lines in $\C^{2}$. 
Suppose $\A$ has $1$ or $2$ lines which contain all 
the intersection points of multiplicity $3$ and higher.  Then 
$\Omega^1_r(\A) =  \sigma_r(\RR^1(\A,\Q))^{\compl}$, 
for all $1\le r\le n$.
\end{prop}

\begin{proof}
As shown by Nazir and Raza in \cite{NR09},  the 
characteristic variety $\VV^1(\A)$ of such an 
arrangement has no translated components.  
The conclusion follows from Theorem \ref{thm:df arr}.
\end{proof}

In general, though, the ``resonance upper bound" for 
the Dwyer--Fried invariants of arrangements is not attained. 
We illustrate this claim with the smallest possible example.

\begin{example}
\label{ex:deleted B3}
Let $\A$ be the deleted $\operatorname{B}_3$ 
arrangement, with defining polynomial 
$Q(\A)=z_0z_1(z_0^2-z_1^2)(z_0^2-z_2^2)(z_1^2-z_2^2)$. 
The jump loci of $\A$ were computed in \cite{S02}.  Briefly, 
the resonance variety $\RR^1(\A)\subset \C^8$ contains $7$ 
local components, corresponding to $6$ triple points and one 
quadruple point, and $5$ non-local components, corresponding 
to braid sub-arrangements.  In particular, $\codim \RR^1(\A)=5$.  

In addition to the $12$ subtori arising from the subspaces  
in $\RR^1(\A)$, the characteristic variety $\VV^1(\A)\subset 
(\C^{\times})^8$ also contains a component of the form 
$\rho\cdot T$, where $T$ is a $1$-dimensional algebraic 
subtorus, and $\rho$ is a root of unity of order $2$.  

Let $X=X(\A)$ be complement of the arrangement. 
Then $X$ is formal, yet not $1$-straight.  
Moreover, the hypothesis of Theorem \ref{thm:tt} are satisfied 
for $X$.  We conclude that $\Omega^1_2(\A)$ is strictly 
contained in $\sigma_2(\RR^1(\A))^{\compl}$. 
\end{example}

\section{Acknowledgements}
\label{sec:ack}

An incipient version of this work was presented at the 
Mathematical Society of Japan Seasonal Institute 
on {\em Arrangements of Hyperplanes}, held at 
Hokkaido University in August 2009. I wish to 
thank the organizers for the opportunity 
to participate in such an interesting meeting, 
and for their warm hospitality.

A fuller version of this work was presented  at the 
Centro di Ricerca Matematica Ennio De Giorgi in Pisa, in 
May--June 2010.  I  wish to thank the organizers of the 
Intensive Research Period on {\em Configuration Spaces: 
Geometry, Combinatorics and Topology}\/ for their friendly 
hospitality, and for providing an inspiring mathematical environment.

Finally, I am grateful to Stefan Papadima for many 
illuminating discussions on the topics presented here, and to 
Graham Denham for help with the proof of Theorem \ref{thm:tt}.


\end{document}